\documentclass{amsart}
 
\usepackage[english]{babel}

\usepackage{mathrsfs}
\usepackage{tikz-cd}
\usepackage[colorlinks]{hyperref}
\hypersetup{allcolors=blue}
\usepackage{indentfirst}
\usepackage{enumerate}
\usepackage{float}
\usepackage{graphicx}
\usepackage{color}
\usepackage{mathtools}
\usepackage{caption}

\usepackage{verbatim}

\usepackage[colorinlistoftodos]{todonotes}

\newtheorem{thmINTRO}{Theorem}

\newtheorem{corINTRO}{Corollary}[thmINTRO]
\newtheorem{thm}{Theorem}[section]
\newtheorem{rmk}[thm]{Remark}
\newtheorem{conj}{Conjecture}
\newtheorem{prop}[thm]{Proposition}
\newtheorem{cor}{Corollary}[thm]
\newtheorem{lema}[thm]{Lemma}
\newtheorem{defi}[thm]{Definition}
\newtheorem{exe}[thm]{Example}
\newtheorem{nonexe}[thm]{Non-Example}

\newtheorem{conv}{Convention}

\usepackage{apptools}

\def\Spec{\text{Spec}}

\def\C{\mathbb{C}}
\def\G{\mathbb{G}}
\def\A{\mathbb{A}}
\def\P{\mathbb{P}}
\def\Q{\mathbb{Q}}
\def\Z{\mathbb{Z}}
\def\F{\mathcal{F}}
\def\Cw{\mathscr{C}_w}
\def\x{x_0^{b_0}x_1^{b_1}x_2^{b_2}x_3^{b_3}x_4^{b_4}}
\def\m{\underline{m}}
\def\k{\kappa}

\everymath{\displaystyle}

\emergencystretch=\maxdimen
\title{Symplectic cohomology of quasihomogeneous
$cA_n$ singularities}
\date{\today}

\author[N. Adaloglou]{Nikolas Adaloglou \textsuperscript{1}}
\address{\textsuperscript{1} Mathematical Institute, Leiden University, The Netherlands}
\email{n.adaloglou@math.leidenuniv.nl}

\author[F. Pasquotto]{Federica Pasquotto \textsuperscript{2}}
\address{\textsuperscript{2} Mathematical Institute, Leiden University, The Netherlands}
\email{f.pasquotto@math.leidenuniv.nl}

\author[A. Zanardini]{Aline Zanardini \textsuperscript{3}}
\address{\textsuperscript{3} Institute of Mathematics, EPFL, Switzerland}
 \email{aline.zanardini@epfl.ch}

\begin{document}

\begin{abstract}
We compute the symplectic cohomology of Milnor fibers of isolated quasihomogeneous $cA_n$
singularities . In addition, we use our computations to distinguish their links as contact manifolds and to provide further evidence to a conjecture of Evans and Lekili.
\end{abstract}

\maketitle

\setcounter{tocdepth}{2}
\tableofcontents

\textwidth=12.5cm
\vskip1cm

\begin{sloppypar}

\section{Introduction}

Given a complex isolated hypersurface singularity $p \in X \subset \A^{n+1}$, one can consider the smooth manifold $L$ of real dimension $2n-1$ obtained by intersecting $X$ with a sufficiently small Euclidean sphere of dimension $2n+1$ centered at $p$. 
This manifold, called the link of the singularity, is the boundary of a symplectic manifold (the Milnor fibre) and it carries a natural contact structure. Such contact structure coincides with the field of complex tangencies and only depends on the analytic germ of $X$ around $p$ as explained, for example, in \cite{Caubel}.

For decades, the interplay between the smooth topology of $L$ and the geometry of $(X,p)$ has been extensively studied from different perspectives, leading to some very remarkable results. For a comprehensive introduction to this topic we refer the reader to Seade's monograph \cite{seade}. The study of the contact topology of $L$, on the other hand, stemmed from the development of powerful (albeit hard to compute) invariants coming from Floer theory. In certain restricted cases, these invariants were successfully used, for example, to distinguish non-isomorphic contact structures on the same underlying smooth manifolds (see e.g. \cite{Ustspheres},\cite{Kwon-vanKoert},\cite{uebelebr},\cite{EL}).

More recently, contributions by several authors have been aimed at shedding light onto the relationship between the algebraic geometry of isolated hypersurface singularities and the contact topology of their links. For instance, the realization of certain birational invariants as contact invariants in the work of McLean
\cite{McLean}, or mirror symmetry statements  as in \cite{futaki-ueda-brieskorn}, \cite{futaki-ueda-typeD} and \cite{lekili-ueda}.

In this paper, we consider yet another manifestation of this intriguing relationship, in the form of a recent conjecture by Evans and Lekili \cite[Conjecture 1.4]{EL}. The conjecture predicts that, for an isolated compound Du Val (cDV) singularity  (Definition \ref{def:cdv}), the existence of a small resolution (Definition \ref{def:small}), which lies at the heart of the minimal model program, can be detected by a Floer-type invariant associated to the Milnor fiber $F$, namely its symplectic cohomology. The precise statement is the following:

\begin{conj}[{\cite[Conjecture 1.4]{EL}}]
Let $F$ be the Milnor fiber
of an isolated cDV singularity. Then the singularity admits a small resolution such that the exceptional set has
$\ell$ irreducible components if and only if $SH^*(F)$ has rank $\ell$ in every negative degree.
\label{conj}
\end{conj}

It turns out that in dimension three, the only isolated Gorenstein singularities that can admit a small resolution are exactly the isolated cDV singularities. It is a result of Miles Reid \cite[Theorem 1.1]{reid1} that these singularities are precisely the Gorenstein terminal threefold singularities. In particular, their links are index positive by \cite{McLean}, hence the symplectic cohomology of their Milnor fibers in negative degrees are contact invariants (see e.g. Theorem \ref{thm:bigrading_invariant} or Appendix \hyperref[app:indexpos]{A}).

Therefore, our goal in this paper is twofold: on the one hand, we want to extend the work of \cite{EL} and compute $SH^*(F)$ for further examples of isolated cDV singularities, thus providing additional evidence supporting Conjecture \ref{conj}; and, on the other hand, we also want to use these computations to distinguish infinitely many different contact structures on the associated links. 
Concerning this second goal, in Theorem \ref{thmE} we prove that the singularities we consider can only have contactomorphic links if they are, up to a holomorphic change of coordinates, related by a smooth deformation (Definition \ref{def:deformation}).

In general, the symplectic cohomology of a (non-compact) Liouville domain has a rich algebraic structure that encodes information about the Reeb dynamics on the manifold, and it is usually quite difficult to compute. Here, we compute the symplectic cohomology of Milnor fibers of a plethora of examples of isolated cDV singularities of type $cA_n$ using the practical algorithm outlined in \cite{EL}. 

All the examples we consider arise from invertible polynomials ( Definition \ref{def:inv}) and the computations are 
 based on the following mirror symmetry argument: given an invertible polynomial $w$,  there is an isomorphism of Gerstenhaber algebras between the symplectic cohomology of the Milnor fiber $F$ of the singularity defined by the dual polynomial $\check{w}$ (Definition \ref{dual}) and the Hochschild cohomology of a certain dg-category associated to $w$. As a consequence, if one knows how to compute the latter, then one can determine the symplectic cohomology of the Milnor fiber. 
 
 It is shown in \cite[Theorem 2.6]{EL} that the above argument applies to a large class of invertible polynomials (see also Theorem \ref{comp}), and that $SH^*(F)$ can be computed accordingly, leading the authors to formulate Conjecture \ref{conj}.

After we first made this manuscript public, a refined version of Conjecture \ref{conj} appeared in \cite[Conjecture 5.3]{UL}. Therefore, we have added an argument to show that this refined conjecture also holds for all the invertible polynomials we consider. This is the content of Theorem \ref{thm:new_conj} in Section \ref{sec:Fermat}.

\subsection{Description of the main results}

The singularities we consider arise as the double suspension of invertible curve singularities. That is, we consider singularities that are described by invertible polynomials of the form $x_1^2+x_2^2+g(x_3,x_4)$, where $g$ itself is an invertible polynomial which is of either chain, loop or Fermat type (see e.g. \cite{HS} for the terminology). 

More precisely, we study (the duals of the) polynomials that belong to one of the following families (where $a,b,c,d,e,f>1$):
\begin{equation}
w_{\text{chain}}^{a,b}\coloneqq x_1^2+x_2^2+x_3^ax_4+x_4^b
\label{example-chain}
\end{equation}
or
\begin{equation}
w_{\text{loop}}^{c,d}\coloneqq x_1^2+x_2^2+x_3^cx_4+x_3x_4^d
\label{example-loop}
\end{equation}
or
\begin{equation}
w_{\text{Fermat}}^{e,f}\coloneqq x_1^2+x_2^2+x_3^e+x_4^f
\label{example-Fermat}
\end{equation}

By \cite[Proposition 1.3]{min-discr} (see also \cite[Theorem 2.8]{kollar-terminal}  and \cite{Yau-Yu}), any (isolated) quasihomogeneous $cA_n$ singularity is formally equivalent to the singularity determined by (the dual of) an invertible polynomial in one of the above families. In particular, throughout the paper we will adopt the following three conventions:

\begin{conv}
    A singularity defined by the dual of a polynomial as in  (\ref{example-chain}), (\ref{example-loop}) or (\ref{example-Fermat}) will be called an invertible $cA_n$ singularity.
\end{conv}

\begin{conv}\label{conv:allow}
    An invertible polynomial as in (\ref{example-chain}), (\ref{example-loop}) or (\ref{example-Fermat}) will be called suspended.
\end{conv}

\begin{conv}
Moreover, we will also abuse the terminology and we will often refer to a polynomial $w$ as in (\ref{example-chain}), (\ref{example-loop}) or (\ref{example-Fermat}), or the dual of, as being of chain, loop or Fermat type, respectively. 
\end{conv}

Using the effective algorithm from \cite{EL}, we first prove Theorems \ref{thmA}, \ref{thmB} and \ref{thmC} below thus providing formulas for the ranks of the symplectic cohomology of Milnor fibers of any invertible $cA_n$ singularity, including those which do not admit a small resolution.

\begin{thmINTRO}[= Theorem \ref{main_general} + Corollary \ref{cor:main_chain}]
 Let $w=w_{\text{chain}}^{a,b}$ be as in (\ref{example-chain}) and let $F$ denote the Milnor fiber of the singularity defined by the dual polynomial $\check{w}$. Then for any integer $k \leq 0$ we have:    
   \begin{align*}
   \dim SH^{2k}(F)&=\dim SH^{2k+1}(F)=\\
   &=\begin{cases}
          \gcd(a-1,b) & \text{if $q=0$}\\
          q & \text{if $1\leq q \leq \min\{a-1,b\}$}\\
          \min\{a-1,b\} & \text{if $\min\{a-1,b\}<q\leq \max\{a-1,b\}$}\\
          (a-1+b)-q & \text{if $\max\{a-1,b\}<q\leq (a-1+b)-1$}
      \end{cases} ,
   \end{align*}
   where $0\leq q < a-1+b$ is such that $(a-1)(1-k)\equiv q \mod (a-1+b)$.
    \label{thmA}
\end{thmINTRO}


\begin{thmINTRO}[= Theorem \ref{main_general_loop} + Corollary \ref{cor:main_loop}]
   Let $w=w_{\text{loop}}^{c,d}$ be as in (\ref{example-loop}) and let $F$ denote the Milnor fiber of the singularity defined by the dual polynomial $\check{w}=w$. Then for any  integer $k \leq 0$ we have:
     \begin{align*}
   \dim SH^{2k}(F)&=\dim SH^{2k+1}(F)=\\
   &=\begin{cases}
          \gcd(c-1,d-1)+1 & \text{if $q=0$}\\
          q+1 & \text{if $1\leq q \leq \min\{c,d\}-1$}\\
          \min\{c,d\} & \text{if $\min\{c,d\}-1<q\leq \max\{c,d\}-1$}\\
          (c+d-2)-q+1 & \text{if $\max\{c,d\}-1<q\leq (c+d-2)-1$}
      \end{cases} ,
   \end{align*}
   where $0\leq q < c+d-2$ is such that $\min\{c-1,d-1\}(1-k)\equiv q \mod (c+d-2)$.
      \label{thmB}
\end{thmINTRO}


\begin{thmINTRO}[= Theorem \ref{main_Fermat} + Corollary \ref{cor:main_Fermat}]\label{thmC}
Let $w=w_{Fermat}^{e,f}$ be as in (\ref{example-Fermat}) and let $F$ denote the Milnor fiber of the singularity defined by the dual polynomial $\check{w}=w$. Then for any integer $k\leq 0$ we have:
\begin{align*}
\dim SH^{2k}(F)&=\dim SH^{2k+1}(F) =\\
&=\begin{cases}
        \gcd(e,f)-1 & \text{if $q=0$} \\
        q-1 & \text{if $1\leq q \leq \min\{e,f\} $} \\
        \min\{e,f\}-1 & \text{if $\min\{e,f\}<q\leq \max\{e,f\}$}\\
        e+f-q-1 & \text{if $\max\{e,f\}<q\leq e+f-1$}
            \end{cases} ,
\end{align*}
where $0\leq q < e+f$ is such that $\min\{e,f\}(1-k)\equiv q \mod (e+f)$.
\end{thmINTRO}

\begin{rmk}
    In \cite[Section 5]{UL}, the authors have also computed $SH^*(F)$ for $w$ of Fermat type. They have not, however, given a closed formula for $\dim SH^{2k}(F)$ when $k\leq 0$.
\end{rmk}

Then, as a first application of our computations, we verify Conjecture \ref{conj} holds for all the three types of invertible polynomials addressed in this paper by making use of the criterion given by Lemma \ref{crit}. That is, we prove:

\begin{thmINTRO}[= Propositions \ref{conjtrue-chain}, \ref{conjtrue-loop} and \ref{conjtrue-Fermat} combined]\label{thmD}
    Conjecture \ref{conj} holds for (the duals of) all the polynomials as in (\ref{example-chain}), (\ref{example-loop}) or (\ref{example-Fermat}).
\end{thmINTRO}

In particular, we establish the following, which complements \cite[Theorem 1.8]{EL}:

\begin{corINTRO}
    Let $F$ be the Milnor fiber of a quasi-homogeneous isolated $cA_n$ singularity which admits a small resolution. Then $\dim SH^{r}(F)=n$ for all $r\leq 1$. 
\end{corINTRO}

In a different direction, we further use Theorems \ref{thmA}, \ref{thmB}, and \ref{thmC} to prove, as already mentioned, that the links of any two invertible $cA_n$ singularities are contactomorphic if and only if the two singularities are deformation
equivalent as in Definition \ref{def:deformation}. In other words, we show invertible $cA_n$ singularities are completely determined, up to smooth deformation, by the contact topology of their links:



\begin{thmINTRO}[= Proposition \ref{cvsf-rho1}]\label{thmE}
   Any two invertible $cA_n$ singularities have contactomorphic links if and only if, up to a holomorphic change of variables, their defining polynomials are deformation equivalent (as in Definition \ref{def:deformation}). 
   
   In particular, this can only happen if both singularities admit a small resolution with the same number of exceptional curves.
\end{thmINTRO}

We achieve this by systematically keeping track of a bigrading on $SH^{\leq 1}(F)$ that comes from the Gerstenhaber bracket and that yields some useful contact invariants of the link. We refer the reader to Section \ref{sec:top} for the details.

Since smooth deformations, as well as holomorphic change of coordinates,  preserve all the classical topological invariants of an isolated hypersurface singularity, we further obtain: 

\begin{corINTRO}[= Corollary \ref{invmilnornumber}]
Two invertible $cA_n$ singularities with contactomorphic links must have the same Milnor number.
\end{corINTRO}

Note that Theorem \ref{thmE} can be used as a means of distinguishing different contact structures on the same underlying smooth $5$-manifold.  For example, the singularities coming from Fermat-type polynomials $w^{e,f}_{Fermat}$ all have links which are diffeomorphic to $\#_{\ell}(S^2\times S^3)$, where $\ell=\gcd(e,f)-1$ and, by convention, $\#_0(S^2\times S^3)=S^5$. Therefore, we can distinguish infinitely many different contact structures on $\#_{\ell}(S^2\times S^3)$:

\begin{corINTRO}[Computations in Subsections   \ref{sec:comparison} and \ref{sec:rho1}]
Two invertible $cA_n$ singularities coming from two suspended polynomials of Fermat type define the same contact structure on $\#_{\ell} (S^2\times S^3$) if and only if they both admit a small resolution and they both have the same Milnor number. 

In particular, any two contact structures on $S^5$ coming from distinct suspended Fermat-type polynomials are never contactomorphic.
\end{corINTRO}

Viewing smooth manifolds as links of singularities to produce new contact structures on them goes back to \cite{Ustspheres}. Closer to our setting is the beautiful work of Uebele in \cite{uebelebr}, where he distinguished the contact structures coming from Fermat-type polynomials for $e=2$, where the links are diffeomorphic to either $S^5$ or $S^2\times S^3$. This was generalized in \cite{EL} in various directions, but only in the case where the singularities considered admit a small resolution.

As a concluding remark, we would like to add that we believe that the significance of Theorem \ref{thmE} and its corollaries goes beyond $cA_n$ singularities, in that they hint towards a stronger connection between symplectic and algebraic geometry of isolated singularities. For instance, one could optimistically view Corollary \ref{invmilnornumber} as a first step towards establishing the uniqueness of Milnor fillings, at least in dimension three. 

Finally, we would also like to mention that while this manuscript was in preparation, we learned that M. Habermann and J. Asplund are also working on a similar project. Their work, however, differs from ours, in that they are trying to find an intrinsically symplectic approach to computing symplectic cohomology of Milnor fibers of $cA_n$ singularities, which does not rely on mirror symmetry arguments.

\subsection{Organization of the paper}

The paper is organized as follows: In Section \ref{sec:back} we introduce the notation and we recall the necessary background notions and results on the algebraic geometry of $cDV$ singularities and the smooth and symplectic topology of their links. We also include a self-contained discussion of the relevant mirror symmetry setup for invertible polynomials. Section \ref{sec:computing} contains the necessary preliminary lemmas and propositions for our computations. In Section \ref{sec:chain} we prove Theorems \ref{thmA}, \ref{thmB} and \ref{thmC} and we explain how one can use these to prove Theorem \ref{thmD}, thus confirming that Conjecture \ref{conj} holds for all invertible $cA_n$ singularities. Finally, in Section \ref{sec:top} we describe how to use the bigrading on $SH^{*}(F)$ to prove Theorem \ref{thmE}.\par
We also include two appendices: In Appendix \hyperref[app:indexpos]{A} we present some generalities on the symplectic cohomology of Liouville domains with an emphasis on index positivity. Appendix \hyperref[app:expcomp]{B}, in turn, contains a few explicit examples that illustrate the applicability of our formulas from Section \ref{sec:chain}. In particular, we explain how to recover \cite[Theorem 3.13]{EL}.

\section*{Acknowledgments}

We would like to thank Jonny Evans, Matthew Habermann, and Chris Peters for insightful conversations and their valuable comments. 
We are also grateful to Yanki Lekili for kindly pointing out a mistake in a previous version of our work and for drawing our attention to the refined conjecture in the manuscript \cite{UL}.

This project originated from a seminar we organized in Leiden during the Fall of 2022 and we would like to warmly thank Emma Brakkee, Chris Peters, George Politopoulos, and Adrien Sauvaget for their participation. Finally, we would like to mention that the comprehensive lecture notes \cite{chris-notes} written by Chris Peters are also based on this seminar. 
\section{Background}\label{sec:back}

In this section, we present the relevant background material that will be needed for our exposition. 

\subsection{Compound Du Val singularities and small resolutions}\label{sub:cDV}

We start with a brief overview of a class of threefold singularities, which are known as compound Du Val singularities, as well as their connection to small resolutions.

\begin{defi}\label{def:cdv}
A \textbf{compound Du Val singularity}, or cDV for short, is a germ of an algebraic variety
(or of an analytic space) $(X,p)$ which is analytically equivalent to the germ of a threefold hypersurface singularity $(\{f=0\},0)$ in the affine space $\mathbb{A}^4$ given by an equation of the form
\begin{equation}
f(x_1,x_2,x_3,x_4)=g_{X_n}(x_1,x_2,x_3)+x_4\tilde{g}(x_1,x_2,x_3,x_4)
\label{cdv}
\end{equation}
where $g_{X_n}=0$ is the equation of a Du Val (surface) singularity (also known as ADE) and $\tilde{g}$ is an arbitrary polynomial.
\end{defi}

 In other words, a cDV singularity is a threefold singularity such that a generic hyperplane section is a Du Val singularity. In particular, borrowing the notation from \cite{min-discr}, in (\ref{cdv}) $X_n$ stands for $A_n,D_n$ or $E_n$ and the polynomial $g_{X_n}$ is one of the following polynomials:

\vspace{-0.5cm}
\begin{align*}
 g_{A_n}&= x_1^2+x_2^2+x_3^{n+1} \quad (n\geq 1)  \\
 g_{D_n}&= x_1^2+x_2^2z+x_3^{n-1} \quad (n\geq 4) \\
 g_{E_6} &= x_1^2+x_2^3+x_3^4 \\
 g_{E_7}&= x_1^2+x_2^3+x_2x_3^3 \\
 g_{E_8}&=x_1^2+x_2^3+x_3^5 
\end{align*}

Moreover, we then say:

\begin{defi}
    A cDV singularity $(X,p)$ is of type $cX_n$ if $X_n$ is minimal in a representation of $(X,p)$ by equation (\ref{cdv}), where the ordering is as in \cite{min-discr}.
\end{defi}

It turns out that, in dimension three, cDV singularities are the only Gorenstein singularities that can admit a so-called small resolution. 

Small resolutions are always crepant, meaning they do not affect the canonical class, and they are characterized by the property of having a small exceptional locus. The precise definition is as follows:

\begin{defi}\label{def:small}
    A resolution of singularities is called small if the exceptional locus has codimension at least two. 
\end{defi}

A basic and well-known example of an isolated cDV singularity that admits a small resolution is given by the conifold singularity:

\begin{exe}[\cite{atiyah}]
 Consider $X: \{x_1x_2-x_3x_4=0\}$ in $\A^4$. Then $X$ has an isolated cDV singularity at the origin of type $cA_1$. The blow-up $Y \to X$ along the plane $x_1=x_4=0$ is a small resolution having an irreducible exceptional curve whose normal bundle (in $Y$) is $\mathcal{O}(-1)\oplus \mathcal{O}(-1)$. The blow-up along the plane $x_2=x_3=0$ is another small resolution, and the two are related by a so-called Atiyah flop.
 \label{conifold}
\end{exe}

In general small resolutions of cDV singularities do not always exist, and for (isolated) $cA_n$ singularities one has the following criterion:

\begin{lema}[{\cite[Theorem 1.1]{katz} and \cite[p. 676]{friedman}}]
Any isolated c$A_n$ singularity admitting a small resolution is given by an equation $x_1^2+x_2^2+g(x_3,x_4)=0$ where the germ of the plane curve $g=0$ at the origin has $n+1$ distinct smooth branches. And, conversely, any such singularity admits a small resolution. 
\label{crit}
\end{lema}

Moreover, as Example \ref{conifold} already shows, when small resolutions do exist, they are not unique. Nonetheless, their exceptional sets are isomorphic and, therefore, it makes sense to talk about the number of its irreducible components.

It is well known that whenever an isolated cDV singularity admits a small resolution such that the exceptional curve has $\ell$ irreducible components, then its link $L$ is diffeomorphic to $\#_{\ell}(S^2\times S^3)$. In particular, $H^2(L)\simeq H_3(L)$ and $H_2(L)\simeq H^3(L)$ are both free of rank $\ell$. The converse, however, is not true.

For example, any $cDV$ singularity that is given by $x_1^2+x_2^2+x_3^e+x_4^f=0$, with $\gcd(e,f)>1$, is such that $H_2(L)\simeq H^3(L)$ is free of rank $\ell=\gcd(e,f)-1$ and $L$ is diffeomorphic to $\#_{\ell}(S^2\times S^3)$ \cite[Theorem 10.3.3]{sasakian}. But, as Lemma \ref{crit} tells us, such singularities do not always admit a small resolution.  

In general, the link of an isolated hypersurface singularity $(X,p)$ is defined as the intersection of $X$ with a small sphere centered at $p$. If $X$ has (complex) dimension $n-1$, then the link is a smooth and compact manifold of (real) dimension $2n-3$, which is $n-3$ connected. 

For the convenience of the reader, we recall next some generalities about links of hypersurface singularities as well as their natural contact structures.

\subsection{The link of a singularity and its contact structure}
\label{subsec:links}

 For simplicity, we may assume $X\subset \A^n$ is given by the zero set of a polynomial function and the (isolated) singular point $p$ is the origin. That is, we consider singular algebraic varieties of the form $V(f)=\{x\in \mathbb{C}^n\,:\ f(x)=0\}$, where $f\in \mathbb{C}[x_1,\ldots,x_n]$.

We then define the \textbf{Milnor number} of the singularity to be the dimension of the algebra of functions modulo the Jacobian ideal of $f$, namely the number 
    \[
\mu(f)=\dim_\mathbb{C}\frac{\mathbb{C}[x_1,\ldots,x_n]}{\left\langle\frac{\partial f}{\partial x_1},\ldots \frac{\partial f}{\partial x_n}\right\rangle}.
\]

And, moreover, we view each $V(f)$ as a limit (degeneration) of a family of varieties
\[
V_t(f)=\{x\in \mathbb{C}^n\,:\,f(x)=t\}
\]
The topology of such families was extensively studied by Milnor (\cite{Milnor}), who showed the following:

\begin{prop}
The preimage of each sufficiently small 
complex number $t$ intersects a small open ball of radius $\delta$ in
a smooth, parallelizable $(2n-2)$-manifold
\[
F_t=B^{2n}(\delta)\cap f^{-1}(t),
\]
called the \textbf{Milnor fiber} of the singularity.  
Moreover, $F_t$ has the homotopy type of a bouquet of $(n-1)$-spheres,
where the number of spheres is given by the Milnor number of the singularity:
\[
F_t\simeq \bigvee_{\mu(f)} S^{n-1}.
\]
\end{prop}

 The link of the singularity is then defined precisely as the boundary of the Milnor fiber:
 
 \begin{defi}\label{def:link}
For sufficiently small $\delta>0$, the intersection 
\[
L_f:=V(f)\cap S^{2n-1}(\delta)
\]
of $V(f)$ with the sphere of radius $\delta$ centered at $0$ is a smooth manifold, whose diffeomorphism type does not depend on $\delta$. We call this manifold the \textbf{link} of the isolated singularity. 
\end{defi}

Now, it turns out that considering at each point of $L_f$ the subspace of the tangent space which is invariant under complex multiplication, we get a codimension one subbundle of the tangent bundle, also called the field of complex tangencies. These subspaces form a contact distribution on $L_f$, which is compatible in a suitable sense with the symplectic (in fact, Stein) structure of the Milnor fiber. In other words, we have:  

\begin{lema}\label{lemma:link_fillable}
    Let $L_f$ be the link of an isolated hypersurface singularity. The complex tangencies distribution defines a contact structure on $L_f$, which is symplectically fillable. 
\end{lema}
\begin{proof}
Notice that $L_f$ is a level set of the function $q(x)=|x|^2$, viewed as a function on $V(f)$. Now, if $J:\mathbb{C}^n\rightarrow \mathbb{C}^n$ denotes multiplication by $i$ on each component and we set $\beta=-dq\circ J$, then we have that the restriction of the form $\beta$ to $L_f$ defines a contact structure $\xi$ on $L_f$ such that $\xi_p=T_pL_f\cap JT_pL_f$.
To prove symplectic fillability one argues as follows: for $t$ sufficiently small, $V_t(f)$ is a symplectic (in fact, Stein) filling of the contact manifold $V_t(f)\cap S^{2n-1}(\delta)$, which in turn is contactomorphic to $L_f$, the link of the singularity.
\end{proof}

For a more detailed account and further results the reader can consult \cite{Caubel}.

\subsection{Symplectic cohomology and its algebraic structure}\label{sec:sympcoh}

As stated in the introduction, our goal is to compute the symplectic cohomology of Milnor fibers of some examples of isolated cDV singularities. With this in mind, we recall below the definition of symplectic cohomology in the general setting of Liouville domains. We explain in Appendix \ref{app:indexpos} how, in the examples we are interested in, this invariant only depends on the contact structure of the link.

We are going to discuss the definition of symplectic cohomology for a class of non-compact manifolds that arise as the completion of compact symplectic manifolds with contact boundaries. 

\begin{defi}
    A \textbf{Liouville domain} is a compact symplectic manifold $(W,\omega)$ with boundary, together with a globally defined Liouville vector field $Y$ (meaning $\mathcal{L}_Y\omega=\omega$) which points transversely out of the boundary.
    \end{defi}
    
    In particular, $\lambda=\iota_Y\omega$ is a primitive for $\omega$ (in other words, $d\lambda=\omega$) and $(M,\omega)$ is an exact symplectic manifold) and $\alpha=\lambda|_{\partial W}$ is a contact form.  
    An important observation is that Milnor fibers and links of isolated singularities fall into this framework: a Liouville vector field is given by the gradient of the function $q(x)=\|x\|^2$ with respect to the standard K\"ahler metric (cf. Lemma \ref{lemma:link_fillable}).
    
    A Liouville domain $(W,\omega)$ can be completed to a non-compact Liouville manifold
\[
\hat{W}=W\cup_{\Sigma}  \Sigma\times [0,+\infty),
\]
by attaching a copy of $\Sigma\times [0,+\infty)$ to $W$ along a collar neighborhood of the boundary, as prescribed by the flow of the Liouville vector field. 

If $\hat{W}$ is the completion of a Liouville domain, then $SH^*(\hat{W})$, the symplectic cohomology of $\hat{W}$, is the cohomology of a chain complex $SC^*(\hat{W},\alpha)$ with generators:

\begin{itemize}
\item $C^*(W)$, the cochain complex of $M$;
\item periodic orbits of the Reeb vector field of $\alpha$ (two for each orbit).
\end{itemize}

For more details on the definition of this invariant, we refer the reader to Seidel's lecture notes \cite{Seidel}.

What we would like to point out here, is that symplectic cohomology is endowed with a rich algebraic structure: whenever $c_1(W)=0$, $SH^*(\hat{W})$ admits the structure of a Gerstenhaber algebra.

A Gerstenhaber algebra is a graded complex vector space $A^*=\bigoplus_k A^k$ equipped with an associative, graded commutative product and a graded Lie algebra bracket $[\ ,\ ]$ of degree $-1$ 
satisfying a mutual compatibility condition. 

An example is given by the Hochschild cohomology of any associative algebra. It has a Gerstenhaber bracket which, together with the cup product, forms a Gerstenhaber algebra structure \cite{gerstenhaber}.

In general, if $A^*$ is a Gerstenhaber algebra, the subspace $A^1$ is a Lie algebra over $\mathbb{C}$ whose Lie bracket gives representations $A^1\rightarrow \mathfrak{gl}(A^k)$ of $A^1$ on $A^k$ for each $k$.  Now, if $\mathfrak{h}$ is a $1$-dimensional Cartan subalgebra of $A^1$ (this always exists and is unique up to automorphism) and we choose an identification of $\mathfrak{h}^*$ with $\mathbb{C}$, 
then the adjoint representation induces a $\mathbb{Z}\times\mathbb{C}$ bigrading -- a \emph{weight space decomposition} of each graded piece
\[
A^k\cong \bigoplus_{\tau\in\mathbb{C}} B_\tau.
\]

It turns out that whenever $W$ is a Liouville domain satisfying $c_1(W)=0$, the symplectic cohomology $SH^*(\hat{W})$ carries a Gerstenhaber bracket 
\[
[\cdot,\cdot]:SH^m\times SH^n\rightarrow SH^{m+n-1}.
\]
and hence we get a decomposition of each $SH^r$ parametrized by $\mathbb{C}$, with the parametrization depending on the identification of $\mathfrak{h}\simeq \mathbb{C}$.
More details on the definition of the bracket can be found in \cite[Section 4.3]{EL}.

The bigrading coming from this Gerstenhaber algebra structure will play an important role in Section \ref{sec:top}.


\subsection{Mirror symmetry for invertible polynomials}\label{sec:mirror}

We now end this background section with a brief discussion on a version of homological mirror symmetry for \textit{invertible polynomials}. In our context, it consists of a conjecture relating the symplectic topology and the algebraic geometry of two polynomials which are mirrors via matrix transposition. Moreover, it provides a method for tackling the computations we are interested in.

We first recall the definition of an invertible polynomial:

\begin{defi}
We say a polynomial 
\begin{equation}\label{inv-pol}
    w=w(x_1,\ldots,x_n)\coloneqq \sum_{i=1}^n \prod_{j=1}^n x_j^{a_{ij}} \in \C[x_1,\ldots,x_n]
\end{equation}
is invertible if and only if the three conditions below hold:
\begin{enumerate}[(i)]
\item The $n\times n$ matrix $A_w=(a_{ij})$ is invertible (over $\Q$);
\item $w$ is quasi-homogeneous, meaning,  there exists a (uniquely determined) weight system $(d_1,\ldots,d_n;h)$, satisfying gcd$(d_1,\ldots,d_n,h)=1$, such that
\[
\sum_{j=1}^n d_j a_{ij}=h \qquad \forall \, i=1,\ldots,n
\]
In particular, $ w(\lambda^{d_1}x_1,\ldots,,\lambda^{d_n}x_n)=\lambda^h w(x_1,\ldots,x_n)$ for all $\lambda \in \C$ ; and
\item $w$ is quasi-smooth, i.e. the hypersurface $H_w: \{w=0\}\subset \A^{n}$ has a (unique) isolated singularity at the origin.
\end{enumerate}
\label{def:inv}
\end{defi}

Next, we recall the definition of the so-called Berglund-H\"{u}bsch mirror (or dual) \cite{BH} to an invertible polynomial. Given $w$ as in Definition \ref{inv-pol}, its mirror, denoted by $\check{w}$, is defined as follows:

\begin{defi}\label{dual}
    If $w$ is an invertible polynomial as in (\ref{inv-pol}), we define $\check{w}$ to be the invertible polynomial associated to $A_{w}^{T}$, the transpose of $A_w$. That is,
    \[
      \check{w}\coloneqq \sum_{i=1}^n \prod_{j=1}^n x_j^{a_{ji}}
    \]
\end{defi}

 Homological mirror symmetry for invertible polynomials predicts that the Fukaya-Seidel category of the mirror polynomial $\check{w}$ is equivalent to the category of matrix factorizations associated with a certain group action on the affine space $\A^n$ that preserves the polynomial $w$. 
 
 More precisely, given $w$ as above,  one considers its maximal symmetry group, which is the following finite extension of $\G_m$ (the multiplicative torus) 
\begin{equation}\label{msg}
    \Gamma_w \coloneqq \left\{\underline{t}=(t_0,t_1,\ldots,t_n)\in (\G_m)^{n+1}; \prod_{j=1}^n t_j^{a_{ij}}=t_0\cdot t_1\cdot \ldots \cdot  t_n \,\, \forall \, i=1,\ldots,n\right\}
\end{equation}
that acts on $\A^n$ (preserving $w$) via $(x_1,\ldots,x_4)\mapsto (t_1x_1,\ldots,t_nx_n)$. Then, to this data, one associates the dg-category $\text{mf}(\A^n,\Gamma_w,w)$ of $\Gamma_w$-equivariant matrix factorizations and one postulates the following:

\begin{conj}[{\cite[Conjecture 1.2]{lekili-ueda}}]
Let  $w\in \C[x_1,\ldots,x_n]$ be an invertible polynomial and write $\check{w}$ for the dual polynomial. Then there exists a quasi-equivalence of idempotent complete $A_{\infty}$-categories $\mathcal{F}(\check{w})\simeq \text{mf}(\A^n,\Gamma_w,w)$, where $\mathcal{F}(\check{w})$ denotes the Fukaya-Seidel category of a Morsification of $\check{w}$.
\label{mirror}
\end{conj}

As shown in \cite{EL}, such conjecture, when it holds, provides a tool for computing the symplectic cohomology of Milnor fibers of isolated hypersurface singularities defined by certain classes of invertible polynomials.
The conjecture has been established in many cases, including:

\begin{thm}(\cite{futaki-ueda-brieskorn}, \cite{futaki-ueda-typeD}, \cite{HS})
    Conjecture \ref{mirror} holds for suspended polynomials.
\label{thm:conj-hms}
\end{thm}

Since $\Gamma_{w}$ also acts naturally on $\A^{n+1}$, one can also consider the corresponding category of equivariant matrix factorizations  $\text{mf}(\A^{n+1},\Gamma_w,w)$, which from now on we will denote by $\Cw$. Then one can prove the following:

\begin{thm}[{\cite[Theorem 2.6]{EL}}]
Let  $w\in \C[x_1,\ldots,x_n]$ be an invertible polynomial and, as before, write $\check{w}$ for the dual polynomial. If $HH^2(\mathscr{C}_{w})=0$, $d_0 \coloneqq h-\sum_{i=1}^4d_i \neq 0$ and Conjecture \ref{mirror} holds, then there exists an isomorphism of Gerstenhaber algebras $SH^*(F) \simeq HH^*(\mathscr{C}_{w})$, between the symplectic cohomology of the Milnor fiber $F$ of the singularity defined by $\check{w}$ and the Hochschild cohomology of $\mathscr{C}_{w}\coloneqq \text{mf}(\A^{n+1},\Gamma_w,w)$.
\label{comp}
\end{thm}

We will show later, in Proposition \ref{hyp}, that all the hypotheses of Theorem \ref{comp} are indeed satisfied for the polynomials we are interested in. In particular, we can compute $SH^*(F)$, together with its Gerstenhaber algebra structure, by computing $HH^*(\mathscr{C}_{w})$ instead.

We will come back to this Gerstenhaber algebra structure, and the bigrading one can extract from it, in subsection \ref{subsec:bigrading}.

\section{The general framework for the computations}\label{sec:computing}

In this section, we explain how to compute $HH^*(\mathscr{C}_{w})$ for an invertible polynomial $w$ of either chain, loop of Fermat type using the formula in \cite[Theorem 1.2]{kernels}  and the practical algorithm outlined in \cite[Section 2.4]{EL}. That is, we consider suspended polynomials $w$, as in Convention \ref{conv:allow}, which belong to one of the families (\ref{example-chain-2}), (\ref{example-loop-2}) or (\ref{example-Fermat-2}) below, respectively. 

\begin{equation}
w_{\text{chain}}^{a,b}\coloneqq x_1^2+x_2^2+x_3^ax_4+x_4^b \qquad (a,b \geq 2)
\label{example-chain-2}
\end{equation}
or
\begin{equation}
w_{\text{loop}}^{c,d}\coloneqq x_1^2+x_2^2+x_3^cx_4+x_3x_4^d \qquad (c,d \geq 2)
\label{example-loop-2}
\end{equation}
or
\begin{equation}
w_{\text{Fermat}}^{e,f}\coloneqq x_1^2+x_2^2+x_3^e+x_4^f \qquad (e,f \geq 2)
\label{example-Fermat-2}
\end{equation}

\subsection{Setting and notation}
We start by introducing the notations which will be used throughout the rest of the paper. We follow \cite{EL} closely and refer the reader to it for more details.

\begin{enumerate}[\protect\fbox{{N}\arabic{enumi}}]
\item We let $\chi_w: \Gamma_w \to \G_m$ be the character defined by 
\[
\underline{t}\coloneqq(t_0,\ldots,t_4) \mapsto  t_0 t_1  t_2  t_3  t_4
\]
where $\Gamma_w$ is the maximal symmetry group of $w$, as in (\ref{msg}).
\item[]
\item[] Then, given $\gamma \in \ker(\chi_w)$, we consider its diagonal action on $\A^5\simeq \Spec(\C[x_0,\ldots,x_4])$, and we let:
\item[]
\item $I\coloneqq \{1,2,3,4\}=I^{\gamma}\cup I_{\gamma}$, where 
\begin{itemize}
\item $i\in I^{\gamma} \iff x_i$ is fixed under the action of $\gamma$,
\item $i\in I_{\gamma} \iff x_i$ is not fixed under the action of $\gamma$;
\end{itemize}
\item $w_{\gamma}$ denote the restriction of $w$ to $\bigcap_{i \in I_{\gamma}} \{x_i=0\}$;
\item $J_{\gamma}$ denote the Jacobian ring of $w_{\gamma}$
\item and $M_\gamma$ be the subset of $\C[x_0,x_1,x_2,x_3,x_4]$ given by $A_{\gamma} \cup B_{\gamma} \cup C_{\gamma}$, where
\begin{itemize}
\item $A_{\gamma}$ is the set of monomials of the form $x_0^{\beta}p\prod_{i\in I_{\gamma}} x^*_i$ if $x_0$ is fixed by $\gamma$; otherwise, we put $A_{\gamma}=\emptyset$. Where, here, $\beta \geq 0$ and $p\in J_{\gamma}$
\item $B_{\gamma}$ is the set of monomials of the form $x_0^{\beta}p x^*_0\prod_{i\in I_{\gamma}} x^*_i$ if $x_0$ is fixed by $\gamma$; otherwise, we put $B_{\gamma}=\emptyset$. And, again, $\beta \geq 0$ and $p\in J_{\gamma}$
\item $C_{\gamma}$ is the set of monomials of the form $px^*_0\prod_{i\in I_{\gamma}} x^*_i$ if $x_0$ is not fixed by $\gamma$; otherwise, we put $C_{\gamma}=\emptyset$. And, once more, $p\in J_{\gamma}$.
\end{itemize}
\end{enumerate}

We then say an element $\underline{m}\in M_{\gamma}$ is a $\gamma$-monomial and we write $b_k$ for the total exponent of $x_k$ in $\underline{m}$, with the convention that $x^*_k$ contributes $-1$ to $b_k$. In particular, we write $\underline{m}=x_0^{b_0}x_1^{b_1}x_2^{b_2}x_3^{b_3}x_4^{b_4}$ and each $\underline{m}\in M_{\gamma}$ then defines a character
$\chi_{\underline{m}}: \Gamma_{w} \to \G_m$ given by 
\[
\chi_{\underline{m}}: \underline{t} \mapsto  t_0^{b_0}t_1^{b_1}t_2^{b_2}t_3^{b_3}t_4^{b_4},
\]
which allows us to further introduce the notion of \textbf{good pairs}:

\begin{defi}
We say a pair $(\gamma,\underline{m})\in \text{ker}(\chi_w) \times M_{\gamma}$ is a good pair if and only if  $\chi_{\underline{m}}=\chi_w^{\otimes N}$ for some integer $N$.
\end{defi}

\begin{rmk}
Elements $\underline{m}\in M_{\gamma}$ actually lie in $\C[x_0,x_1,x_2,x_3,x_4]\otimes \Lambda^{\dim N_{\gamma}} (N_{\gamma})^{\vee}$, where $N_{\gamma}$ is the vector space spanned by the non-fixed variables $x_i$ and the generators of $(N_{\gamma})^{\vee}$ are denoted by $x^*_i$.  But we then adopt the convention that $x^*_i$ contributes $-1$ to $b_i$ as above.
\end{rmk}

\begin{rmk}\label{b0}
Note that given any $\underline{m}=\x\in M_{\gamma}$ we have $b_0\geq -1$.
\end{rmk}

We then have:

\begin{thm}[{\cite[Theorem 2.14]{EL}}]
The only contributions to $HH^*(\mathscr{C}_{w})$ come precisely from good pairs $(\gamma,\underline{m})$. Moreover, a good pair $(\gamma,\underline{m})$ contributes to: 
\begin{itemize}
\item $HH^{2N+3-\alpha+1}$ if $\underline{m}\in A_{\gamma}$
\item $HH^{2N+3-\alpha+2}$ if $\underline{m}\in B_{\gamma}$
\item $HH^{2N+3-\alpha+2}$ if $\underline{m}\in C_{\gamma}$
\end{itemize}
where $\alpha\coloneqq |I^{\gamma}|$.
\label{mainTHM}
\end{thm}

In particular, we are now in a position to prove that all invertible polynomials we are interested in do satisfy the hypotheses of Theorem \ref{comp}, hence we have an isomorphism between $HH^*(\mathscr{C}_{w})$ and $SH^*(F)$ -- the symplectic cohomology of the Milnor fiber of the singularity defined by $\check{w}$. 

Concretely, we can now prove:

\begin{prop}
If $w$ is a suspended polynomial, then $HH^2(\mathscr{C}_{w})=0$ and Conjecture \ref{mirror} holds (Theorem \ref{thm:conj-hms}). Moreover, if $w$ has weights $(d_1,d_2,d_3,d_4;h)$, then the integer $d_0\coloneqq h-\sum_{i=1}^4 d_i$ is not equal to zero.   
\label{hyp}
\end{prop}

\begin{proof}
First, we observe $HH^2(\mathscr{C}_{w})=0$ (see also Lemma \ref{lem:deg2}). In fact, by Theorem \ref{mainTHM}, contributions to $HH^2(\mathscr{C}_{w})$ can only come from pairs $(\gamma,\underline{m})$ such that $\alpha=|I^{\gamma}|=4$ and $\chi_{\underline{m}}=\chi_{w}$ (i.e. $N=1$). Therefore, potentially, $\gamma\in \ker(\chi_w)$ could only be the identity element and $\underline{m}$ could only be the monomial $=x_0x_1x_2x_3x_4$. However, looking at the ring $J_{\gamma}$ we see that $\underline{m}\notin M_{\gamma}$.

Finally, we note that for a polynomial $w$ as in the statement, we have that $d_1=d_2=h/2$, hence $0< d_3+d_4=-d_0$. 
\end{proof}

In practice, to carry out the computations of $HH^r(\mathscr{C}_{w})$, for $r\neq 2$ and for the polynomials we are interested in, one needs a more functional description of the group $\text{ker}(\chi_w)$. 

Therefore, if $w$ is as in (\ref{example-chain}) or (\ref{example-loop}), let  $G=\mu_{d_w}$ with $d_w\coloneqq (\det A_{w})/4$; and if $w$ is as in (\ref{example-Fermat}), let $G=\mu_e\times\mu_f$. We obtain such a functional description by considering a covering homomorphism 
\[
\Psi:\mu_2\times\mu_2\times G \times \G_m\rightarrow \Gamma_w,
\]
where throughout $\mu_{\bullet}$ denotes the cyclic group of roots of unity of order $\bullet$, which we view as a subgroup of $\mathbb{G}_m$. 

Moreover, we can then make use of Lemma \ref{inducedcharlemma} below, repeatedly and implicitly, to describe the corresponding good pairs:

\begin{lema}\label{inducedcharlemma}
A character $\xi:\Gamma_w\rightarrow \mathbb{G}_m$ is of the form $\chi^{\otimes N}=\xi$ if and only if $(\chi\circ \Psi)^{\otimes N}=\xi\circ \Psi$. 
\end{lema}
\begin{proof}
If $\chi^N=\xi$, then $\chi^N\circ \Psi=\xi\circ \Psi$ and it suffices to observe we have $(\chi\circ \Psi)^N=\chi^N\circ \Psi$. Conversely, if $(\chi\circ \Psi)^N=\xi\circ \Psi$, then we obtain  $\chi^N\circ \Psi=\xi\circ \Psi$, where again we use the equality $(\chi\circ \Psi)^N=\chi^N\circ \Psi$. Now, because $\Psi$ is surjective, it has a right inverse, and $\chi^N\circ \Psi=\xi\circ \Psi$ implies that $\chi^N=\xi$.
\end{proof}

\subsection{Determining the good pairs}\label{detker}

Let us now describe, in great detail, how to determine the good pairs contributing to $HH^{\ast}(\Cw)$ for a polynomial $w$ as in (\ref{example-chain}) or (\ref{example-loop}). For Fermat-type polynomials the procedure has already been elaborated in \cite[Section 3.1]{EL} and the description can be summarized in the following proposition:

\begin{prop}[{\cite[Section 3.1]{EL}}]\label{contFermat}
If $w=w^{e,f}_{Fermat}$ is as in (\ref{example-Fermat}), then 
\begin{enumerate}[(i)]
\item the contributions to $HH^3(\mathscr{C}_{w})$ come precisely from good pairs $(\gamma,\underline{m})$ such that $\underline{m}$ has total exponent $b_0=-1$, hence $\underline{m} \in B_{\gamma}\cup C_{\gamma}$. Furthermore, $\dim HH^3(\Cw)=(e-1)(f-1)$.
\item for any $r\leq 1$ the contributions to $HH^r(\mathscr{C}_{w})$ come precisely from good pairs $(\gamma,\underline{m})$ such that $\underline{m}$ has total exponent $b_0\geq 0$ and $\gamma$ is of the form $(1,1,\zeta,\xi)$ or $(-1,-1,\zeta,\xi)$, where $\zeta\in \mu_{e}$ and $\xi \in \mu_f$ are such that $\zeta\cdot \xi=1\in \mathbb{G}_m$ (i.e. $\gamma$ fixes the variable $x_0$). And, conversely, any such good pair can only contribute to $HH^r(\mathscr{C}_{w})$  for some $r\leq 1$ 
\item it follows from the previous point that $HH^r(\mathscr{C}_{w})$ vanishes for all $r>3$.
\end{enumerate}
\end{prop}

Our main goal here will thus be to prove an analogous statement when $w$ is either of chain or of loop type. Concretely, we will prove Proposition \ref{hh3} (the analogue of $(i)$) and Proposition \ref{negcont} (the analogue of $(ii)$ and $(iii)$).

Therefore, henceforth let $w$ be as in (\ref{example-chain}) or (\ref{example-loop}), let $(d_1,d_2,d_3,d_4;h)$ be the corresponding weights and denote the entries of the corresponding matrix $A_{w}$ as in Definition \ref{inv-pol} by $(a_{ij})$. And note that if $w$ is as in (\ref{example-chain}) (resp. (\ref{example-loop})), then $A_w$ is as in $(\star)$ (resp. $(\star\,\star)$) below:

\[
 (\star)\underset{\text{chain type}}{\begin{pmatrix}
2 & 0 & 0 & 0\\
0 & 2 & 0 & 0\\
0 & 0 & a & 1 \\
0 & 0 & 0 & b
 \end{pmatrix}} \qquad
 (\star\,\star)\underset{\text{loop type}}{\begin{pmatrix}
2 & 0 & 0 & 0\\
0 & 2 & 0 & 0\\
0 & 0 & c & 1 \\
0 & 0 & 1 & d 
 \end{pmatrix}} 
\]

Now, to describe the good pairs, we will consider the following surjective $h:1$ covering homomorphism 
\[
\Psi: (u_1,u_2,u_3,\tau)\mapsto (\tau^{d_0}u_1^{-1}u_2^{-1}u_3^{a_{33}-1},\tau^{d_1}u_1,\tau^{d_2}u_2,\tau^{d_3} u_3,\tau^{d_4} u_3^{-a_{33}})
\]

which gives an isomorphism $\text{ker}(\chi_w\circ \Psi) \simeq \mu_2\times \mu_2 \times \mu_{d_w}\times \mu_h$, hence an isomorphism 
\[
\text{ker}(\chi_w) \simeq \mu_2\times \mu_2 \times \mu_{d_w},
\] 
where (as before) $d_w\coloneqq (\det A_{w})/4$.

In particular, we have that the action of each $\gamma=(u_1,u_2,u_3) \in \text{ker}(\chi_w)$ on $\A^5\simeq \Spec(\C[x_0,\ldots,x_4])$ is diagonal and it is given by 
\[
(x_0,x_1,x_2,x_3,x_4)\mapsto (u_1^{-1}u_2^{-1}u_3^{a_{33}-1} x_0,u_1 x_1,u_2 x_2,u_3 x_3,u_3^{-a_{33}} x_4)
\]

Moreover, if we fix $\gamma \in \text{ker}(\chi_w)$ and a $\gamma$-monomial $\underline{m}=x_0^{b_0}x_1^{b_1}x_2^{b_2}x_3^{b_3}x_4^{b_4}$, then using the map $\Psi$ we can identify $\chi_{\underline{m}}$ with the character
\[
(u_1,u_2,u_3,\tau)\mapsto \tau^{n_0} u_1^{n_1} u_2^{n_2} u_3^{n_3}
\]
where the integers $n_i$ are given by:

\begin{itemize}
    \item $n_0=\sum_{i=0}^4 b_i d_i$,
    \item $n_1=b_1-b_0$,
    \item $n_2=b_2-b_0$ and
    \item $n_3=-a_{33}b_4+b_3+(a_{33}-1)b_0$
\end{itemize}

And, similarly, we can identify $\chi_w$ with the character 
\[
(u_1,u_2,u_3,\tau)\mapsto \tau^{h}
\] 

Consequently, we have that $\chi_{\underline{m}}=\chi_w^{\otimes N}$ for some $N\in \Z$ if and only if $n_0= N\cdot h$ for some $N\in \Z$ and we can find integers $m_1,m_2,m_3$ such that
\begin{equation}\label{mi}
    n_1=2m_1 \qquad n_2=2m_2 \qquad n_3=d_wm_3
\end{equation}

Thus, 
\begin{equation}\label{cong-ab}
    -a_{33}b_4+b_3+(a_{33}-1)b_0\equiv 0 \mod d_w,
\end{equation}

$(\gamma,\underline{m})$ is a good pair (by definition) and
\begin{equation}\label{Nform}
\begin{split}
N &=m_1+m_2+(a_{44}-1)m_3+b_4 \\ 
 &=\frac{b_1+b_2}{2}+(a_{44}-1)m_3+b_4-b_0
\end{split}
\end{equation}

This proves the following:

\begin{prop}\label{cont}
The only contributions to $HH^*(\mathscr{C}_{w})$ come, possibly, from elements $\gamma \in \ker(\chi_w)$ of the form $(1,1,\zeta_{d_w})$ or $(-1,-1,\zeta_{d_w})$, where $\zeta_{d_w}\in \mu_{d_w}$. 
\end{prop}

\begin{proof}
    If $(\gamma,\underline{m})$ is a good pair with $\underline{m}=x_0^{b_0}x_1^{b_1}x_2^{b_2}x_3^{b_3}x_4^{b_4}$, then  our discussion tells us $b_1 \equiv b_2 \mod 2$. Thus, from the description of $M_{\gamma}$ and $J_{\gamma}$ we see that we must have $b_1=b_2=0$ or $b_1=b_2=-1$. Equivalently, either both $x_1$ are $x_2$ are fixed, or both variables are not fixed. Therefore, $\gamma$ is of the form $(1,1,\zeta_{d_w})$ or $(-1,-1,\zeta_{d_w})$. 
\end{proof}

Observe now that an element $\gamma=(u_1,u_2,u_3) \in \ker(\chi_w)$ will never fix $x_0$, unless $u_3^{a_{33}-1}=1$. Therefore, we also have:

\begin{prop}
 If $(\gamma,\underline{m})$ is a good pair such that $\gamma$ fixes $x_0$, then $\gamma$ is of the form $(1,1,\zeta)$ or $(-1,-1,\zeta)$ for some $\zeta \in \mu_{a_{33}-1}$.
\label{cont2}
\end{prop}

At the same time, if $(\gamma,\underline{m})$ is a good pair such that $\gamma$ does not fix $x_0$, hence $\underline{m} \in C_{\gamma}$, then it turns out such a pair contributes to $HH^3(\mathscr{C}_{w})$. More precisely, we can further prove the following:

\begin{prop}
The contributions to $HH^3(\mathscr{C}_{w})$ come precisely from good pairs $(\gamma,\underline{m})$ such that $\underline{m}$ has total exponent $b_0=-1$, hence $\underline{m} \in B_{\gamma}\cup C_{\gamma}$. Furthermore,
\begin{enumerate}[(i)]
    \item $\dim HH^3\left(\Cw)\right)=a(b-1)+1$ if $w=w_{\text{chain}}^{a,b}$ is as in (\ref{example-chain}), and
    \item $\dim HH^3\left(\Cw\right)=cd$ if $w=w_{\text{loop}}^{c,d}$ is as in (\ref{example-loop}).
\end{enumerate}

\label{hh3}
\end{prop}

\begin{proof}
First, observe that by Theorem \ref{mainTHM}, contributions to $HH^3(\mathscr{C}_{w})$ can only come from pairs $(\gamma,\underline{m})$ such that either:
\begin{enumerate}[(i)]
\item $\alpha=|I^{\gamma}|=0,\underline{m}\in B_{\gamma}\cup C_{\gamma}$ and $\chi_{\underline{m}}=(\chi_{w})^{-1}$ (i.e. $N=-1$); or
\item $\alpha=|I^{\gamma}|=2,\underline{m}\in B_{\gamma}\cup C_{\gamma}$ and $\chi_{\underline{m}}=(\chi_{w})^{0}$ (i.e. $N=0$); or
\item $\alpha=|I^{\gamma}|=1,\underline{m}\in A_{\gamma}$ and $\chi_{\underline{m}}=(\chi_{w})^{0}$ (i.e. $N=0$).
\end{enumerate}

The first case implies $\underline{m}=x_0^{b_0}x_1^{-1}x_2^{-1}x_3^{-1}x_4^{-1}$ for some $b_0\geq -1$ since $3,4\in I_{\gamma}$ and $J_{\gamma}=\C$. And then equation (\ref{Nform}) gives $b_0=-1$. A similar argument shows the second case also implies $b_0=-1$ and that the third case does not happen.

Conversely, if $b_0=-1$, then the descriptions of $M_{\gamma}$ and $J_{\gamma}$ combined with (\ref{mi}) imply that a pair $(\gamma,\underline{m})$ can only be a good pair if  $\underline{m} \in B_{\gamma}\cup C_{\gamma}$ and 
\[
\underline{m}=x_0^{-1}x_1^{-1}x_2^{-1}x_3^{b_3}x_4^{b_4},
\]
where by (\ref{cong-ab}) the exponents $b_3$ and $b_4$ are such that either $b_4=b_3=-1$, hence $N=-1$ and $\alpha=0$; or $b_4=0$ and $b_3=a-1$ and hence $N=0$ and $\alpha=2$. In particular, $m_1=m_2=m_3=0$ and either $\{3,4\}\subset I_{\gamma}$ or $\{3,4\} \subset I \backslash I_{\gamma}$.

Therefore, the contributions from monomials with total exponent $b_0=-1$ come from $\gamma=(-1,-1,1)$ or from $\gamma=(-1,-1,\zeta)$ for some 
$\zeta \in \mu_{d_w}\backslash \mu_{a_{33}}$. Note that there are $a(b-1)+1$ such elements if $w$ is as in (\ref{example-chain}), and there are $cd$ such elements if $w$ is as in (\ref{example-loop}). Moreover, in each case, these indeed contribute to $HH^3(\mathscr{C}_{w})$. 
\end{proof}

\begin{rmk}
    Note that in Proposition \ref{hh3}, the dimension of $HH^3(\mathscr{C}_{w})$  is precisely the Milnor number of the singularity defined by $\check{w}$. In fact, it follows from (\ref{les}) that $SH^3(F)\simeq H^3(F)$ (see also Corollary \ref{cor:milnumb}).
\end{rmk}

Finally, putting everything together, we obtain:

\begin{prop}
For any $r\leq 1$ the contributions to $HH^r(\mathscr{C}_{w})$ come precisely from good pairs $(\gamma,\underline{m})$ such that $\underline{m}$ has total exponent $b_0\geq 0$ and $\gamma$ is of the form $(1,1,\zeta)$ or $(-1,-1,\zeta)$ for some $\zeta\in \mu_{a_{33}-1}\cap \mu_{d_w}$ (i.e. $\gamma$ fixes $x_0$). Conversely, any such good pair can only contribute to $HH^r(\mathscr{C}_{w})$  for some $r\leq 1$. In particular, $HH^r(\mathscr{C}_{w})$ vanishes for all $r>3$.
\label{negcont}
\end{prop}

\begin{proof}
Let $(\gamma,\underline{m})$ be a good pair which contributes to $HH^r(\mathscr{C}_{w})$  for some $r\leq 1$. Then $\underline{m}$ has total exponent $b_0\geq 0$ by Remark \ref{b0} and Proposition \ref{hh3}. In particular, $\gamma \notin C_{\gamma}$ and $\gamma$ fixes $x_0$. Thus, $\gamma$ is of the form  $(1,1,\zeta)$ or $(-1,-1,\zeta)$ for some $\zeta\in \mu_{a_{33}-1}\cap \mu_{d_w}$ by Propositions \ref{cont} and \ref{cont2}. 

Conversely, let $(\gamma,\underline{m})$ be any good pair such that $\underline{m}$ has total exponent $b_0\geq 0$ and $\gamma$ is of the form $(1,1,\zeta)$ or $(-1,-1,\zeta)$ for some $\zeta\in \mu_{a_{33}-1}\cap \mu_{d_w}$. Assume such a pair contributes to $HH^r(\mathscr{C}_{w})$ for some $r\in \Z$. Then, since $\alpha=I^{\gamma}$ is always even, it follows from Theorem \ref{mainTHM} that 
\[
r=\begin{cases}
    2N-\alpha+4 & \text{if $r$ is even}\\
    2N-\alpha+5 & \text{if $r$ is odd}    
\end{cases} 
\]
Moreover, writing $\underline{m}=x_0^{b_0}x_1^{b_1}x_2^{b_2}x_3^{b_3}x_4^{b_4}$ we further have
\[
2N-\alpha+4=\frac{2}{h}\cdot (d_0b_0+d_3b_3+d_4b_4) \qquad (\iff \zeta=1)
\]
or
\[
2N-\alpha+4=2+ \frac{2}{h}\cdot(d_0b_0+d_3b_3+d_4b_4) \qquad (\iff \zeta\neq 1)
\]
Now, in the first case, since $d_3+d_4=-d_0$ we conclude
\[
d_0b_0+d_3b_3+d_4b_4\leq (d_0+d_3+d_4)\max\{b_0,b_3,b_4\}=0
\]
And, in the second case, since $b_3=b_4=-1, b_0\geq 0$ and $d_0<0$, we obtain
\[
d_0b_0+d_3b_3+d_4b_4=d_0(b_0+1)< 0
\]
Therefore, in any case, we must have $r\leq 1$.
\end{proof}

\begin{cor}\label{notdivide}
   If $w$ is as in (\ref{example-chain}) or (\ref{example-loop})  and $(\gamma,\underline{m})$ is a good pair which contributes to $HH^r(\mathscr{C}_{w})$ for some $r\leq 1$, then $(\gamma,\underline{m})$  belongs to one of the following cases:
\begin{enumerate}[(i)]
    \item $\gamma=(1,1,1)$ and $\underline{m}=x_0^{b_0}p$ where  $p\in \C[x_3,x_4]/(\partial_{x_3} w, \partial_{x_4} w)$
    \item $\gamma=(-1,-1,1)$ and $\underline{m}=x_0^{b_0}px_1^{-1}x_2^{-1}$ where  $p$ is as in $(i)$
    \item $\gamma=(1,1,\zeta)$, with $\zeta\neq 1$, and $\underline{m}=x_0^{b_0}x_3^{-1}x_4^{-1}$
    \item $\gamma=(-1,-1,\zeta)$, with $\zeta\neq 1$, and $\underline{m}=x_0^{b_0}x_1^{-1}x_2^{-1}x_3^{-1}x_4^{-1}$
\end{enumerate}
In particular, if gcd$(a_{33}-1,d_w)=1$, then $(\gamma,\underline{m})$ is as in $(i)$ or $(ii)$.
\end{cor}

\begin{rmk}\label{ringJ-chain}
    Note that if $w=w_{\text{chain}}^{a,b}$ is as in (\ref{example-chain}) and $\gamma \in \ker(\chi_w)$ fixes both $x_3$ and $x_4$, then the Jacobian ring $J_{\gamma}=\C[x_3,x_4]/(\partial_{x_3} w, \partial_{x_4} w)$ has basis 
    \[
\left\{1,x_3,\ldots,x_3^{2(a-1)},x_4,x_4^2,\ldots,x_4^{b-2}\right\}\cup\left\{x_3^ix_4^j\right\}_{1\leq i\leq a-2, 1\leq j\leq b-2}.
\]
\end{rmk}

\begin{rmk}\label{ringJ-loop}
     Similarly, if $w=w_{\text{loop}}^{c,d}$ is as in (\ref{example-loop})  and $\gamma\in \ker(\chi_w)$ is such that both $x_3$ and $x_4$ are fixed, then the Jacobian ring 
     $J_{\gamma}= \C[x_3,x_4]/(\partial_{x_3} w, \partial_{x_4} w)$ has basis 
     \[
     \left\{1,x_3,\ldots,x_3^{c-1},x_4,x_4^2,\ldots,x_4^{2(d-1)}\right\}\cup\left\{x_3^ix_4^j\right\}_{1\leq i\leq c-2, 1\leq j\leq d-1}.
     \]
 \end{rmk}

\section{The main formulas}\label{sec:chain}

We will now use Theorem \ref{mainTHM} combined with Propositions \ref{contFermat} and \ref{negcont} to deduce the main formulas in Theorems \ref{thmA}, \ref{thmB} and \ref{thmC}. That is, we will compute  $HH^{\leq 1}(\mathscr{C}_{w})$ for $w$ a suspended polynomial.

In addition, we will use these formulas to determine for which parameters $a$ and $b$ (resp. $c$ and $d$ or $e$ and $f$) the polynomial $w$ is such that the cohomology groups $HH^{\leq 1}(\mathscr{C}_{w})$ have constant rank. Then, using Lemma \ref{crit}, we will show (Propositions \ref{conjtrue-chain}, \ref{conjtrue-loop} and \ref{conjtrue-Fermat}) that this happens precisely when the corresponding dual singularity admits a small resolution, thus proving Conjecture \ref{conj} for all suspended polynomials.

We will consider the three types of polynomials separately, but our approach is quite uniform. If $w$ is of chain type (resp. of loop type), we observe that to each good pair $(\gamma,\m)$ contributing to $HH^r(\Cw)$ we can associate the integer $m_3$ that is uniquely determined by equations (\ref{mi}) and (\ref{Nform}), and that this integer must lie in one of the sets appearing in our formula (\ref{formula-chain}) (resp. (\ref{formula-loop})) below. Therefore, computing  $HH^r(\Cw)$ can be reduced to the counting of the number of integers $m_3$ satisfying certain integer (in)equalities.

If $w$ is of Fermat type, we apply a similar argument: the only difference in this case is that good pairs $(\gamma,\m)$ correspond to pairs of integers $(m_3,m_4)$.

This approach turns out to be quite useful because it allows us to keep track of the exponent $b_0$ in a contributing monomial $\m=\x$. This is precisely the information one needs when determining the bigrading on $HH^{\ast}(\Cw)$ (see Lemma \ref{lem:b0}).

\subsection{The chain-type polynomials}\label{subsec:general}

We first explain how one can compute $HH^r\left(\mathscr{C}_{w}\right)$ for $r\leq 1$, $w=w_{\text{chain}}^{a,b}$ as in (\ref{example-chain}), and   $2\leq a$ and $2\leq b$ arbitrary. 

We prove:

\begin{thm}\label{main_general}
   Let $w=w_{\text{chain}}^{a,b}$ be as in (\ref{example-chain}). Then, given any integer $k \leq 0$, we have: 
   \begin{equation}\label{formula-chain}
      \dim HH^{2k}(\mathscr{C}_{w})=\dim HH^{2k+1}(\mathscr{C}_{w})=|\mathcal{W}_k|+|\mathcal{X}_k|+\eta|\mathcal{Y}_k|+\sum_{i=1}^{b-2}|\mathcal{Z}_{i,k}|,
   \end{equation}
       where 
   \begin{itemize}
        \item $\mathcal{W}_k\coloneqq \{m_3\in \Z_{\leq -k}\,;\, 1\leq-(a+b-1)m_3-(a-1)k\leq b-2\}$,
        
        \item $\mathcal{X}_k\coloneqq \{m_3\in \Z_{\geq 0}\,;\, 0\leq (a+b-1)m_3+(a-1)k\leq 2(a-1)\}$  ,
        
        \item  $\mathcal{Y}_k\coloneqq \{m_3\in \Z_{\geq 0}\,;\, (a+b-1)m_3=(a-1)(1-k)\}$ ,
        
        \item $\mathcal{Z}_{i,k} \coloneqq \{m_3\in \Z\,;\, 1\leq (a+b-1)m_3+i+(a-1)k\leq a-2\}$ and
        
        \item $\eta \coloneqq |\mu_{a-1}\cap \mu_{ab}|-1=\gcd(a-1,b)-1$.
   \end{itemize}
   
\end{thm}

\begin{proof}
Fix a pair $(\gamma,\underline{m}) \in \ker(\chi_w)\times M_{\gamma}$ and write $\underline{m}=x_0^{b_0}x_1^{b_1}x_2^{b_2}x_3^{b_3}x_4^{b_4}$. Then it follows from 
    Theorem \ref{mainTHM}, Corollary \ref{notdivide} and Remark \ref{ringJ-chain} that  $(\gamma,\underline{m})$ contributes to  $HH^{2k}(\mathscr{C}_{w})$ if and only if one of the following cases holds:
    \begin{enumerate}[($i$)]
        \item  $\gamma=(1,1,1)$, $b_1=b_2=0$ and $\underline{m}=x_0^{b_0}x_3^{b_3}x_4^{b_4}\in A_{\gamma}$ is such that  $b_0\geq 0,b_3$ and $b_4$ satisfy 
    \[
    -ab_4 + b_3+(a-1)b_0=abm_3
    \]
    for some integer $m_3$ and either
    \begin{enumerate}[({$i$}-a)]
        \item $b_3=0$ and $1\leq b_4 \leq b-2$, or
        \item $b_4=0$ and $0\leq b_3 \leq 2(a-1)$, or
        \item $1\leq b_3 \leq a-2$ and $1\leq b_4 \leq b-2$
    \end{enumerate}
        \item $\gamma=(-1,-1,1)$, $b_1=b_2=-1$ and $\underline{m}=x_0^{b_0}x_1^{-1}x_2^{-1}x_3^{b_3}x_4^{b_4}\in A_{\gamma}$ is such that $b_0,b_3$ and $b_4$ are as in $(i)$,
         \item $\gamma=(1,1,\zeta)$, $b_1=b_2=0,b_3=b_4=-1$, with $1 \neq \zeta \in \mu_{a-1}\cap \mu_{ab}$, and $\underline{m}=x_0^{b_0}x_3^{-1}x_4^{-1}$,  where $b_0\geq 0$ satisfies $(a-1)(b_0+1)=abm_3$ for some integer $m_3$
    \item $\gamma=(-1,-1,\zeta)$, $b_1=b_2=b_3=b_4=-1$,  with $1 \neq \zeta \in \mu_{a-1}\cap \mu_{ab}$, and $\underline{m}=x_0^{b_0}x_1^{-1}x_2^{-1}x_3^{-1}x_4^{-1}$, where $b_0$ is as in $(iii)$.
    \end{enumerate}

    Now, if $(\gamma,\underline{m})$ is as in case $(i)$ or $(ii)$ above, then equation (\ref{Nform}) gives us
\begin{equation}
   2k= 2N-\alpha+4=2\big((b-1)m_3+b_4-b_0\big)
\end{equation}
And if  $(\gamma,\underline{m})$ is as in case $(iii)$ or $(iv)$ we have
\begin{equation}
   2k= 2N-\alpha+4=2\big((b-1)m_3-b_0\big)
\end{equation}

Therefore, to find the contributions to $HH^{2k}(\mathscr{C}_w)$, we need to look for integers $m_3$ such that either

\begin{enumerate}[($i'$)]
    \item  $1\leq -(a+b-1)m_3-(a-1)k\leq b-2$,  or 
    \item  $0\leq (a+b-1)m_3+(a-1)k\leq 2(a-1)$,  or
    \item $ 1\leq (a+b-1)m_3+i+(a-1)k\leq a-2$ for some $1\leq i \leq b-2$, or 
    \item $(a+b-1)m_3=(a-1)(1-k)$.
\end{enumerate}

Indeed, if we can find an integer $m_3$ as in $(i')$, then letting $b_3=0$, \linebreak $b_4=-(a+b-1)m_3-(a-1)k$ and 
\begin{equation}\label{b0b1b2}
b_0=(b-1)m_3+b_4-k \qquad b_1=b_2=0 \,(\text{resp.}\,=-1)\,\text{if $b_0$ is even (resp. $odd$)}     
\end{equation}
we will have found a good pair contributing to $HH^{2k}(\Cw)$ provided $b_0\geq 0$. And this last condition simply means $m_3\in \mathcal{W}_k$. 

Similarly, if we can find $m_3$ as in $(ii')$ (resp. $(iii')$) or, equivalently, if we can find $m_3\in \mathcal{X}_k$ (resp. $m_3\in \mathcal{Z}_{i,k}$),  then we let $b_4=0$, $b_3=(a+b-1)m_3+(a-1)k$ (resp. $b_4=i$ and $b_3=(a+b-1)m_3+i+(a-1)k$) and, again, $b_0,b_1$ and $b_2$ should be as in (\ref{b0b1b2}). And if we find $m_3$ as in $(iv)$, hence $m_3\in \mathcal{Y}_k$,  we let $b_3=b_4=-1$, $b_0=(b-1)m_3-k$ and $b_1=b_2=0$ (resp. $b_1=b_2=-1$) if $b_0$ is even (resp. odd). Note that in this case, each integer $m_3$ will give a contribution whose dimension is precisely the number $\eta$, which is the number of choices of $\zeta$.

This concludes the proof since the same reasoning also applies for odd degrees replacing $2k$ by $2k+1$, $A_{\gamma}$ by $B_{\gamma}$, etc.
\end{proof}

\begin{rmk}\label{rmk:empty-chain}
We observe that in Theorem \ref{main_general} when $a=2$ or $b=2$ some of the inequalities defining the sets $\mathcal{W}_k,\mathcal{X}_k,\mathcal{Y}_k$ and $\mathcal{Z}_{i,k}$ may become void. We simply mean that in those cases the corresponding sets are empty and there are no associated contributions to be considered.
\end{rmk}

\begin{rmk}\label{Y-chain}
    We also note that $\mathcal{Y}_k\neq \emptyset$ if and only if the non-positive integer $k$ is such that $(a-1)(1-k)\equiv 0\mod (a+b-1)$, in which case $|\mathcal{Y}_k|=1$.
\end{rmk}

Next, using formula (\ref{formula-chain}) above we can further deduce the following:

\begin{cor}\label{cor:main_chain}
    Let $w=w_{\text{chain}}^{a,b}$ be as in (\ref{example-chain}) and let $k\in \Z_{\leq 0}$. Then 
   \begin{equation}\label{formula-chain-simpler}
      \dim HH^{2k}(\mathscr{C}_{w})= \begin{cases}
          \gcd(a-1,b) & \text{if $q=0$}\\
          q & \text{if $1\leq q \leq \min\{a-1,b\}$}\\
          \min\{a-1,b\} & \text{if $\min\{a-1,b\}<q\leq \max\{a-1,b\}$}\\
          (a-1+b)-q & \text{if $\max\{a-1,b\}<q\leq (a-1+b)-1$}
      \end{cases}, 
   \end{equation}
   where $0\leq q < a-1+b$ is such that $(a-1)(1-k)\equiv q \mod (a-1+b)$.
\end{cor}

\begin{proof}
Let us fix $k$ and $q$ as in the statement, and let us write 
\[
(a-1)(1-k)=n(a+b-1)+q
\]
for some $n\in\mathbb{Z}_{\geq 0}$. In view of Theorem \ref{main_general} and Remark \ref{Y-chain},  in order for us to deduce formula (\ref{formula-chain-simpler}) we can argue as follows:

\begin{enumerate}[{Case} 1]
    \item If $q=0$, then by the assumption on $k$,  given any integer $m_3$ we can write
 \[(a+b-1)m_3+(a-1)k=(a+b-1)(m_3-n)+(a-1).\]
In particular, we see that $\mathcal{W}_k=\emptyset=\mathcal{Z}_{i,k}$ and that $\mathcal{X}_k=\{n\}=\mathcal{Y}_k$. Thus, by Theorem \ref{main_general} we have exactly $\eta+1=gcd(a-1,b)$ contributions to $HH^{2k}$ (and to $HH^{2k+1}$). 
    
     \item Assume now $q=\min\{a-1,b\}$. First, when $b\leq a-1$ we have $\mathcal{X}_{k}=\{n,n+1\}$ and $\mathcal{Z}_{i,k}=\{n\}$ for $1\leq i\leq b-2$, while $\mathcal{Y}_k=\mathcal{W}_k=\emptyset$. Thus, we have a total of $b$ contributions. \par

Next, when $a-1\leq b-1$ we have $\mathcal{X}_k=\{n\}=\mathcal{Z}_{i,k}$ for $1\leq i\leq a-2$ and all other sets are empty yielding a total of $a-1$ contributions.\par

Therefore, in any case, we conclude from Theorem \ref{main_general} that the number of contributions is precisely $\min\{a-1,b\}=q$.

    \item Similarly, if $1\leq q \leq \min\{a-1,b\}-1$, then $\mathcal{Y}_k=\emptyset$, and it is routine to check that $|\mathcal{X}_k\cup \mathcal{W}_k|=1$ and that $|\mathcal{Z}_{i,k}|=1$ exactly for $1\leq i \leq q-1$ (and $\mathcal{Z}_{i,k}=\emptyset$ otherwise). Thus, $\sum_{i=1}^{b-2}|\mathcal{Z}_{i,k}|=q-1$ and by Theorem \ref{main_general} we have exactly $q$ contributions to $HH^{2k}$ (and to $HH^{2k+1}$).
      
       \item Next, if $\min\{a-1,b\}<q\leq (a-1+b)-2$, then $\mathcal{Y}_k=\emptyset$ and we have to consider four subcases:

       \begin{enumerate}[{Case 4}(a)]
           \item If $a-1<q<b$, then $b>2$, $\mathcal{X}_k=\emptyset$, $|\mathcal{W}_k|=1$ and either:
         \begin{itemize}
             \item $a=2$,  hence $\gcd(a-1,b)=1$, and in this case $\mathcal{Z}_{i,k}=\emptyset$ for all $1\leq i \leq b-2$ (see also Remark \ref{rmk:empty-chain}); or
             \item $a\geq 3$ and $|\mathcal{Z}_{i,k}|=1$ exactly for $q-(a-2)\leq i \leq q-1$. 
         \end{itemize}
        
                    \item If $a-1<q=b$, then $b=q=\delta(a-1)$ for some $\delta\in \Z_{\geq 1}$ and either:
                    \begin{itemize}
             \item $a=2$, $|\mathcal{X}_k|=1, \mathcal{W}_k=\emptyset$ and the $\mathcal{Z}_{i,k}$ are all empty  ; or
             \item $a\geq 3$, $|\mathcal{X}_k|=2, \mathcal{W}_k=\emptyset$ and $|\mathcal{Z}_{i,k}|=1$ exactly for 
             \[
             b-(a-2)\leq i \leq b-2
             \] 
         \end{itemize}
         
                    \item If $b<q \leq a-1$, then $a\geq 4$, $\mathcal{W}_k=\emptyset$,  $|\mathcal{X}_k|=2$ and either:
                    
           \begin{itemize}
             \item $b=2$  and the $\mathcal{Z}_{i,k}$ are all empty; or
             \item $b\geq 3$ and $|\mathcal{Z}_{i,k}|=1$ for all $1\leq i \leq b-2$. 
         \end{itemize}
         
         \item If $\max\{a-1,b\} < q$, then  $|\mathcal{X}_k|=|\mathcal{W}_k|=1$ and either:
         \begin{itemize}
             \item $q=a+b-3$ and the $\mathcal{Z}_{i,k}$ are all empty; or
             \item $q<a+b-3$ and $|\mathcal{Z}_{i,k}|=1$ exactly for 
             \[
             1\leq i \leq (a-1+b)-q-2
             \] 
         \end{itemize}
         
                  \end{enumerate}
            Therefore, by Theorem \ref{main_general}, we conclude  that there are exactly $\min\{a-1,b\}$ contributions to $HH^{2k}$ (and to $HH^{2k+1}$) when \[
            \min\{a-1,b\}<q\leq \max\{a-1,b\}
            \]
            And  when $\max\{a-1,b\}<q\leq (a-1+b)-2$ we have a total of $(a-1+b)-q$ contributions. 
          \item Finally, if $q=(a-1+b)-1$, then $\mathcal{X}_k=\{n+1\}$ and all other sets appearing in (\ref{formula-chain})  are empty.
\end{enumerate}
   Combining all of our conclusions from the five cases above with Remark \ref{Y-chain} yields formula (\ref{formula-chain-simpler}).
\end{proof}

In particular, the only missing piece we need to verify Conjecture \ref{conj} holds for chain-type polynomials is a numerical criterion that determines when the dual singularity admits a small resolution. This is the content of the next lemma:

\begin{lema}\label{gcd-chain}
    Let $w=w_{\text{chain}}^{a,b}$ be as in (\ref{example-chain}). Then the singularity defined by $\check{w}$ admits a small resolution whose exceptional curve has (exactly) $\min\{a-1,b\}$ irreducible components if and only if $\min\{a-1,b\}=gcd(a-1,b)$. 
\end{lema}

\begin{proof}
    By Lemma \ref{crit}, the singularity defined by $\check{w}$ admits a small resolution whose exceptional curve has (exactly) $n-1$ irreducible components if and only the curve singularity $x_3(x_3^{a-1}+x_4^{b})=0$ has exactly $n=\min\{a-1,b\}+1$ distinct smooth irreducible components. The latter is true if and only $\min\{a-1,b\}=gcd(a-1,b)$ which can be readily checked by writing
    \[
    x_3^{a-1}+x_4^b=\prod_{\zeta,\zeta^{n-1}=-1}\left(x_3^{\frac{a-1}{n-1}}+\zeta x_4^{\frac{b}{n-1}}\right).
    \]
    \end{proof}

Therefore, putting all of the above together we obtain (see also Theorem \ref{thm:new_conj}):

\begin{prop}\label{conjtrue-chain}
If $w=w_{\text{chain}}^{a,b}$ is as in (\ref{example-chain}), then Conjecture \ref{conj} holds for the singularity defined by $\check{w}$.
\end{prop}

\begin{proof}
    It follows from Lemma \ref{gcd-chain}, Theorem \ref{main_general} and Corollary \ref{cor:main_chain} that if the singularity defined by $\check{w}$ admits a small resolution, then $HH^{\leq 1}(\Cw)\simeq SH^{\leq 1}(F)$ has constant rank $\ell=\min\{a-1,b\}$, equal to the number of exceptional curves in such resolution. Conversely, if we have constant rank $\ell$, then it follows from Theorem \ref{main_general} and Corollary \ref{cor:main_chain} that $\min\{a-1,b\}=gcd(a-1,b)=\ell$. Thus, by Lemma \ref{gcd-chain}, the singularity defined by $\check{w}$ admits a small resolution.
\end{proof}

\subsection{The loop-type polynomials}\label{sec:loop}

We will now study polynomials of loop type. Again, we begin by providing a general formula that computes the ranks of $HH^{\ast}(\mathscr{C}_{w})$ for degrees smaller than one.

\begin{thm}\label{main_general_loop}
Let $w=w^{c,d}_{loop}$ be as in (\ref{example-loop}). Then for any  integer $k \leq 0$ we compute 
     \begin{equation}\label{formula-loop}
        \dim HH^{2k}(\mathscr{C}_{w})=\dim HH^{2k+1}(\mathscr{C}_{w})=|\tilde{\mathcal{W}}_k|+|\tilde{\mathcal{X}}_k|+\tilde{\eta}|\tilde{\mathcal{Y}}_k|+\sum_{j=1}^{d-1}|\tilde{\mathcal{Z}}_{j,k}|,
         \end{equation}
   where 
   \begin{itemize}
   \item $\tilde{\mathcal{W}}_k\coloneqq \{m_3\in \Z_{\leq -k}\,;\, 1\leq-(c+d-2)m_3-(c-1)k\leq 2(d-1)\}$,
       \item $\tilde{\mathcal{X}}_k\coloneqq \{m_3\in \Z_{\geq 0}\,;\, 0\leq (c+d-2)m_3+(c-1)k\leq c-1\}$,
              \item  $\tilde{\mathcal{Y}}_k\coloneqq \{m_3\in \Z_{\geq 0}\,;\, (c+d-2)m_3=(c-1)(1-k)\}$,
           \item $\tilde{\mathcal{Z}}_{j,k} \coloneqq \{m_3\in \Z\,;\, 1\leq (c+d-2)m_3+j+(c-1)k\leq c-2\}$ and
         \item $\tilde{\eta} \coloneqq |\mu_{c-1}\cap \mu_{cd-1}|-1=\gcd(c-1,d-1)-1$.
   \end{itemize}
   \end{thm}

\begin{proof}
Fix $k\leq 0$. To find the contributions to $HH^{2k}(\mathscr{C}_w)$ (and to $HH^{2k+1}(\mathscr{C}_w)$), we claim that we can use an argument that is completely analogous to the one in the proof of Theorem \ref{main_general}. 

Indeed, if we consider an integer $m_3\in \tilde{\mathcal{W}}_k$, then letting $b_1=b_2=0$ (resp. $=-1$) if $b_0$ is even (resp. $odd$), and letting $b_3=0$, $b_4=-(c+d-2)m_3-(c-1)k$ and $b_0=(d-1)m_3+b_4-k$, we obtain a good monomial $\underline{m}=x_0^{b_0}x_1^{b_1}x_2^{b_2}x_3^{b_3}x_4^{b_4}$ that contributes to $HH^{2k}(\mathscr{C}_w)$ ( and to $HH^{2k+1}(\mathscr{C}_w)$). Such a monomial will form a good pair with $\gamma=(1,1,1)$ (resp. $\gamma=(-1,-1,1)$) if $b_0$ is even (resp. $odd$).

Similarly, if we can find $m_3\in \tilde{\mathcal{X}}_k$  (resp. $m_3\in \tilde{\mathcal{Z}}_{j,k}$), then we let $b_4=0$, $b_3=(c+d-2)m_3+(c-1)k$ (resp. $b_4=j$ and $b_3=(c+d-2)m_3+j+(c-1)k$) and, again, $b_0,b_1$ and $b_2$ should be as before. This defines a good monomial that will form a good pair with $\gamma=(1,1,1)$ if $b_0$ is even, or with $\gamma=(-1,-1,1)$ if $b_0$ is odd. Again, these will contribute to $HH^{2k}(\mathscr{C}_w)$ and $HH^{2k+1}(\mathscr{C}_w)$.

Lastly, if we find $m_3 \in \tilde{\mathcal{Y}}_k$, we let $b_3=b_4=-1$, $b_0=(d-1)m_3-k$ and $b_1=b_2=0$ (resp. $b_1=b_2=-1$) if $b_0$ is even (resp. odd) thus obtaining a good monomial. The difference in this last case is that, for each $1\neq \zeta\in \mu_{c-1}\cap \mu_{cd-1}$, $\underline{m}=\x$  will form a good pair with $\gamma=(1,1,\zeta)$ (resp. $\gamma=(-1,-1,\zeta)$) instead.

This concludes the proof since Theorem \ref{mainTHM}, Corollary \ref{notdivide} and Remark \ref{ringJ-loop} tell us these will, in fact, exhaust all possible good monomials. 
\end{proof}

\begin{rmk}\label{order}
Here it is important to observe that since the loop polynomials $w_1=w^{c,d}_{loop}$ and $w_2=w^{d,c}_{loop}$ are biholomorphic, it is clear that $\dim HH^{2k}(\mathscr{C}_{w_1})=\dim HH^{2k}(\mathscr{C}_{w_2})$. While at first glance the sets appearing in Formula (\ref{formula-loop}) are not symmetric with respect to $c$ and $d$, their cardinality is. This follows from the fact that if $\gamma_i\in \ker(\chi_{w_i})$ is such that both $x_3$ and $x_4$ are fixed, then the Jacobian rings $J_{\gamma_1}$ and $J_{\gamma_2}$  both have as basis
\[
     \left\{1,x_3,\ldots,x_3^{c-1},x_4,x_4^2,\ldots,x_4^{2(d-1)}\right\}\cup\left\{x_3^ix_4^j\right\}_{1\leq i\leq c-2, 1\leq j\leq d-1}.
     \]
     Thus, choosing this basis for $J_{\gamma_2}$, and using the same argument as in the proof of Theorem \ref{main_general_loop}, we see that $\dim HH^{2k}(\mathscr{C}_{w_2})=\dim HH^{2k+1}(\mathscr{C}_{w_2})$ is also given by $(\ref{formula-loop})$ for any $k\leq 0$ -- without swapping $c$ and $d$. In particular, we can always assume $d\geq c$.
\end{rmk}

\begin{rmk}\label{Y-loop}
    We also observe that, as in the previous case, some of the inequalities in the definition of the sets appearing in (\ref{formula-loop}) become void if $c=2$, and we simply mean that the corresponding sets are empty. Moreover, $\tilde{\mathcal{Y}}_k\neq \emptyset$ if and only if $k$ is such that $(c-1)(1-k)\equiv 0\mod (c+d-2)$, in which case $|\tilde{\mathcal{Y}}_k|=1$.
\end{rmk}

The following statement, which follows directly from formula (\ref{formula-loop}), is the analog of Corollary \ref{cor:main_chain}. It will be used to prove that Conjecture \ref{conj} holds for all loop-type polynomials.

\begin{cor}\label{cor:main_loop}
     Let $w=w_{\text{loop}}^{c,d}$ be as in (\ref{example-loop}) and let $k \in \Z_{\leq 0}$. Then 
   \begin{equation}\label{formula-loop-simpler}
      \dim HH^{2k}(\mathscr{C}_{w})= \begin{cases}
          \gcd(c-1,d-1)+1 & \text{if $q=0$}\\
          q+1 & \text{if $1\leq q \leq \min\{c,d\}-1$}\\
          \min\{c,d\} & \text{if $\min\{c,d\}-1<q\leq \max\{c,d\}-1$}\\
          (c+d-2)-q+1 & \text{if $\max\{c,d\}-1<q\leq (c+d-2)-1$}
      \end{cases}, 
   \end{equation}
   where $0\leq q < c+d-2$ is such that $\min\{c-1,d-1\}\cdot (1-k)\equiv q \mod (c+d-2)$.
\end{cor}


\begin{proof}
    We omit the proof here as the argument is essentially the same as the one in the proof of Corollary \ref{cor:main_chain}. The interested reader is referred to Proposition \ref{prop:formula_loop} in Appendix \hyperref[app:expcomp]{B} instead. 
\end{proof}

In particular, we conclude (see also Theorem \ref{thm:new_conj}):

\begin{prop}\label{conjtrue-loop}
 If $w=w_{\text{loop}}^{c,d}$ is as in (\ref{example-loop}), then Conjecture \ref{conj} holds for the singularity defined by $\check{w}=w$.
\end{prop}

\begin{proof}
   Arguing exactly as in the proof of Lemma \ref{gcd-chain} we can prove that if $w=w_{\text{loop}}^{c,d}$ is as in (\ref{example-loop}), then the singularity defined by $\check{w}$ admits a small resolution whose exceptional curve has (exactly) $\min\{c-1,d-1\}+1$ irreducible components if and only if $\min\{c-1,d-1\}=gcd(c-1,d-1)$. Therefore, the result follows from Theorem \ref{main_general_loop} and Corollary \ref{cor:main_loop}.   
\end{proof}

\subsection{The Fermat-type polynomials}\label{sec:Fermat}

In \cite{EL}, the authors have already computed the ranks of $HH^{\leq 1}(\Cw)$ for invertible polynomials $w=w^{e,f}_{\text{Fermat}}=x_1^2+x_2^2+x_3^e+x_4^f$ in the case where one of the exponents, $f$ or $e$, is a multiple of the other. Here we extend their computations proving Theorem \ref{main_Fermat} below and we further establish that Conjecture \ref{conj} holds for all Fermat-type polynomials as well. 

Similar to the previous cases, the idea we explore to compute  $HH^{\leq 1}(\Cw)$ consists in associating to each good pair $(\gamma,\m=\x)$ contributing to $HH^{r}(\Cw)$ a uniquely determined pair of integers $(m_3,m_4)$ which depends on the degree $r\leq 1$.

Concretely, we prove:

\begin{thm}\label{main_Fermat}
Let $w=w^{e,f}_{Fermat}$ be as in (\ref{example-Fermat}). Given any integer $k\leq 0$ we have:

\begin{equation}\label{formula-Fermat}
    \dim HH^{2k}(\mathscr{C}_w)=\dim HH^{2k+1}(\mathscr{C}_w)=(gcd(e,f)-1)|\F^{k-1}_{-1,-1}| +  \sum_{\substack{0\leq i \leq e-2 \\ 0\leq j \leq f-2}}|\F^k_{i,j}|,
    \end{equation}
where we define 
\[
\F^{s}_{i,j}\coloneqq \{(m_3,m_4)\in \Z\times \Z\,;\,i+m_3e=j+m_4f\geq 0\,\,\text{and}\,\,m_3+m_4=-s\}.
\]
\end{thm}


\begin{proof}
If $w=x_1^2+x_2^2+x_3^e+x_4^f$ is as in (\ref{example-Fermat}), then one can show (we refer to \cite[Section 3]{EL} for the details) that
\[
\ker(\chi_w)=\mu_2\times\mu_2\times\mu_e\times\mu_f.
\]
Moreover, given any good monomial $\m=\x$ we have that the corresponding character $\chi_{\m}$ satisfies $\chi_{\m}=\chi_w^{\otimes N}$ for $N=b_0-\sum_{i=1}^4 m_i$, where the integers $m_i$ are uniquely determined by the following equations
\begin{eqnarray}
   b_0=b_1+2m_1= b_2+2m_2 \\
    b_0=b_3+em_3= b_4+fm_4 \label{b3b4Fermat}
\end{eqnarray}

In particular, if we fix a pair $(\gamma,\underline{m}=\x) \in \ker(\chi_w)\times M_{\gamma}$,  then it follows from 
    Theorem \ref{mainTHM}, the description of $J_{\gamma}$ and Proposition \ref{contFermat} that  $(\gamma,\underline{m})$ contributes to  $HH^{r}(\mathscr{C}_{w})$ if and only if one of the following cases holds:

    \begin{enumerate}[({$i$}-a)]
    \item $\gamma=(1,1,1,1)$, $b_1=b_2=0$ and $\underline{m}=x_0^{b_0}x_3^{b_3}x_4^{b_4}\in A_{\gamma}\cup B_{\gamma}$ is such that $b_0\geq 0$ is even, $b_0,b_3$ and $b_4$ satisfy (\ref{b3b4Fermat}) for some pair of integers $(m_3,m_4)$, and we further have $0\leq b_3 \leq e-2$ and $0\leq b_4 \leq f-2$.
    \item $\gamma=(-1,-1,1,1)$, $b_1=b_2=-1$ and $\underline{m}=x_0^{b_0}x_1^{-1}x_2^{-1}x_3^{b_3}x_4^{b_4}\in A_{\gamma}\cup B_{\gamma}$ is such that $b_0\geq 0$ is odd, and $b_0,b_3$ and $b_4$ are as in ($i$-a).
    \end{enumerate}
    \begin{enumerate}[({$ii$}-a)]
    \item  $\gamma=(1,1,\zeta,\xi)$, $b_1=b_2=0$ and $\underline{m}=x_0^{b_0}x_3^{-1}x_4^{-1}\in A_{\gamma}\cup B_{\gamma}$ is such that $b_0\geq 0$ is even, and $b_0,b_3=-1$ and $b_4=-1$ satisfy (\ref{b3b4Fermat}) for some pair of integers $(m_3,m_4)$
    \item  $\gamma=(-1,-1,\zeta,\xi)$, $b_1=b_2=-1$ and $\underline{m}=x_0^{b_0}x_1^{-1}x_2^{-1}x_3^{-1}x_4^{-1}\in A_{\gamma}\cup B_{\gamma}$ is such that $b_0\geq 0$ is odd, and $b_0,b_3$ and $b_4$ are as in ($ii$-a).
    \end{enumerate}
    where above $\zeta\in \mu_e\backslash\{1\}$  and $\xi \in \mu_f \backslash \{1\}$ must satisfy $\zeta\cdot \xi =1$, since the variable $x_0$ must be fixed by Proposition \ref{contFermat}.
   
    Now observe that $2N=b_1+b_2-2(m_3+m_4)$ and recall that by Theorem \ref{mainTHM} we have
    \[
    r=\begin{cases}
        2N-\alpha+4 & \text{if $\m \in A_{\gamma}$}\\
        2N-\alpha+5 & \text{if $\m \in B_{\gamma}$}
    \end{cases},
    \]
    where $\alpha=|I^{\gamma}|$ is the number of variables among $x_1,\ldots,x_4$ which are fixed by the action of $\gamma$.
    It thus follows that if $(\gamma,\m)$ is as in ($i$-a) or ($i$-b), then $r=-2(m_3+m_4)$ if $\m\in A_{\gamma}$ and $r=-2(m_3+m_4)+1$ if $\m\in B_{\gamma}$. Similarly, if $(\gamma,\m)$ is as in ($ii$-a) or ($ii$-b), then $r=-2(m_3+m_4)+2$ if $\m\in A_{\gamma}$ and $r=-2(m_3+m_4)+3$ if $\m\in B_{\gamma}$.


Therefore, we see that to find the contributions to $HH^{r}(\mathscr{C}_w)$ we need to look for pairs of integers $(m_3,m_4)$ such that either
\begin{enumerate}[($i'$)]
     \item $b_3+m_3e=b_4+m_4f$ for some $0\leq b_3 \leq e-2$ and some $0\leq b_4 \leq f-2$, or  
     \item $-1+m_3e=-1+m_4f$
    \end{enumerate}

We are interested in the solutions to the above equalities which lie precisely in one of the sets $\F^s_{i,j}$ appearing in (\ref{formula-Fermat}), where $s=r/2$ if $r$ is even and $s=(r-1)/2$ if $r$ is odd.

Indeed, if we can find $(m_3,m_4)$ as in $(ii')$, then letting $b_0=-1+m_3e$ and $b_1=b_2=0$ (resp. $b_1=b_2=-1$) it follows that $\m=x_0^{b_0}x_1^{b_1}x_2^{b_2}x_3^{-1}x_4^{-1}$ will form a good pair with $\gamma=(1,1,\zeta,\xi)$ (resp. $\gamma=(-1,-1,\zeta,\xi)$ ) provided $b_0$ is even (resp. odd) and non-negative, which is equivalent to requiring that $(m_3,m_4)\in \F^{s}_{-1,-1}$. Moreover, if $(m_3,m_4)\in \F^s_{-1,-1}$, then the pair $(\gamma,\m)$ we described will contribute to $HH^{2(s+1)}(\Cw)$ and to $HH^{2(s+1)+1}(\Cw)$ and we see that we obtain this way  ($\gcd(e,f)-1$) one-dimensional contributions, namely one for each possible choice of  $\zeta\in \mu_e\backslash\{1\}$  and $\xi \in \mu_f \backslash \{1\}$ satisfying $\zeta\cdot \xi =1$.  

The same kind of reasoning also applies to $(i')$ and, in particular, we obtain the counting in (\ref{formula-Fermat}), since Proposition \ref{contFermat} guarantees the pairs $(\gamma,\m)$ coming from $(i')$ and $(ii')$ indeed exhaust all possible good pairs. 
\end{proof}

\begin{rmk}
Note that in the formula (\ref{formula-Fermat}) we are only interested in the cardinalities of the sets $\F^s_{i,j}$. Therefore, we could have instead simply considered the sets (also depending on $s,i$ and $j$): $\{m_3\in \Z\,;\,i+m_3e=j-(s+m_3)f\geq 0\}$. The choice we made, however, better reflects the argument we use in the proof of the formula.   
\end{rmk}

\begin{rmk}
    Note that we can only find a pair of integers $(m_3,m_4)$ such that $i+m_3e=j+m_4f$ precisely when $i-j$ is a multiple of gcd$(e,f)$. Therefore, in (\ref{formula-Fermat}) we have that $\F^s_{i,j}=\emptyset$ whenever gcd$(e,f)$ does not divide $i-j$. Moreover, if $\F_{i,j}^s\neq\emptyset$, then $|\F_{i,j}^s|=1$ since elements of $\F_{i,j}^s$ can be described as the intersection of two (distinct) lines. 
\end{rmk}

\begin{rmk}
The parametrization of the contributions to $HH^{2k}(\Cw)$ in \cite{UL} agrees with the one above: one has ${\rm{I}}^{k}_{e,f}=\bigsqcup_{\substack{0\leq i \leq e-2 \\ 0\leq j \leq f-2}} \F^k_{i,j}$ and $\mathbb{I}^k_{e,f}=\F^{k-1}_{-1,-1}$.
\end{rmk}

As in the previous two cases, we now deduce:




\begin{cor}\label{cor:main_Fermat}
    Let $w=w^{e,f}_{Fermat}$ be as in (\ref{example-Fermat}) and let $k\in \Z_{\leq 0}$. Then
\begin{equation}\label{formula-Fermat-simpler}
    \dim HH^{2k}(\mathscr{C}_w)= \begin{cases}
        \gcd(e,f)-1 & \text{if $q=0$} \\
        q-1 & \text{if $1\leq q \leq \min\{e,f\} $} \\
        \min\{e,f\}-1 & \text{if $\min\{e,f\}<q\leq \max\{e,f\}$}\\
        e+f-q-1 & \text{if $\max\{e,f\}<q\leq e+f-1$}
            \end{cases}, 
\end{equation}
where $0\leq q < e+f$ is such that $\min\{e,f\}\cdot(1-k)\equiv q \mod (e+f)$. 
\end{cor}

\begin{proof}
Without loss of generality, we may assume that $\min\{e,f\}=e$.\par

Now, first note that if $0\leq q < e+f$ is such that $e(1-k)\equiv q \mod (e+f)$, then $\F_{-1,-1}^{k-1}\neq \emptyset$ exactly when $q=0$. In fact, if $(m_3,m_4)\in \F^{k-1}_{-1,-1}$, then $m_3=1-k-m_4$ and $e(1-k)=m_4(e+f)$.
   
   Next, observe that if $(m_3,m_4)\in \F^{k}_{i,j}$, then $i-j=e-q$ since
 \[
 e(1-k)=e+j-i+m_4(e+f)
 \]
 and $0<2\leq e+j-i\leq e+f-2 < e+f$. Therefore, we have:

\[
\sum_{\substack{0\leq i \leq e-2 \\ 0\leq j \leq f-2}}|\F^k_{i,j}|=\begin{cases}
        0 & \text{if $q=0$} \\
        q-1 & \text{if $1\leq q \leq e $} \\
        e-1 & \text{if $e<q\leq f$}\\
        e+f-1-q & \text{if $f<q\leq e+f-1$}
            \end{cases}. 
            \]

Indeed, 
\begin{enumerate}[{Case} 1]
    \item If $2\leq q \leq e=\min\{e,f\} $, then $\F^k_{i,j}\neq \emptyset$ exactly for
    \[
    (i,j)\in\{(e-q,0),(e-(q-1),1),\ldots,(e-2,q-2)\}
    \]
    \item Similarly, if $e<q\leq f=\max\{e,f\}$, then $\F^k_{i,j}\neq \emptyset$ exactly for
    \[
    (i,j)\in\{(0,q-e),(1,q-(e-1)),\ldots,(e-2,q-2)\}
    \]
    \item And if $f<q<e+f-1$, then $\F^k_{i,j}\neq \emptyset$ exactly for
    \[
    (i,j)\in\{(e-2-(q-f),f-2),(e-3-(q-f),f-3),\ldots,(0,f-(e+f-q))\}
    \]
    \item Finally, further note that when $q\in \{0,1,e+f-1\}$, then $\F^{k}_{i,j}=\emptyset$ for all $0\leq i\leq e-2$ and for all $0\leq j\leq f-2$. 
    \end{enumerate}  
\end{proof}

And we further obtain (see also Theorem \ref{thm:new_conj} below):

\begin{prop}\label{conjtrue-Fermat}
 If $w=w_{\text{Fermat}}^{e,f}$ is as in (\ref{example-Fermat}), then Conjecture \ref{conj} holds for the singularity defined by $\check{w}=w$.
\end{prop}

\begin{proof}
   Once more we can argue exactly as in the proof of Lemma \ref{gcd-chain} and prove that if $w=w_{\text{Fermat}}^{c,d}$ is as in (\ref{example-Fermat}), then the singularity defined by $\check{w}$ admits a small resolution whose exceptional curve has (exactly) $\min\{e-1,f-1\}$ irreducible components if and only if $\min\{e,f\}=gcd(e,f)$. Therefore, the statement follows from Corollary \ref{cor:main_Fermat}.   
   \end{proof}

Using our computations we can also address the refined version of Conjecture \ref{conj}, which appears in  \cite[Conjecture 5.3]{UL}, and show that it holds for all the polynomials we consider in this manuscript. This follows from the next result, which extends \cite[Proposition 5.1]{UL} to chain and loop-type polynomials:

\begin{prop}\label{prop:ranks}
    Let $w$ be a suspended polynomial with corresponding matrix $A_w=(a_{ij})$ and let $k\in \Z_{\leq 0}$. Define $\tilde{w}$ to be the Fermat-type polynomial 
    \begin{equation}\label{pol_qfac}
        x_1^2+x_2^2+x_3^{\tilde{e}}+x_4^{\tilde{f}},
    \end{equation}
    
    where $\tilde{e}=(a_{33}-a_{34})/g_w$, $\tilde{f}=(a_{44}-a_{43})/g_w$ with $g_w=\gcd(a_{33}-a_{34},a_{44}-a_{43})$.
    
    Depending on the type of $w$, the following holds for the rank of $HH^{2k}(\Cw)$:
    \begin{enumerate}[(i)]
        \item If $w$ is of chain type, then
        \[
        \dim\left(HH^{2k}(\Cw)\right)=g_w\left(\dim\left(HH^{2k}(\mathscr{C}_{\tilde{w}})\right)+1\right)
        \]
        \item If $w$ is of loop type, then
        \[
        \dim\left(HH^{2k}(\Cw)\right)-1=g_w\left(\dim\left(HH^{2k}(\mathscr{C}_{\tilde{w}})\right)+1\right)
        \]
        \item If $w$ is of Fermat type, then
        \[
        \dim\left(HH^{2k}(\Cw)\right)+1=g_w\left(\dim\left(HH^{2k}(\mathscr{C}_{\tilde{w}})\right)+1\right).
        \]
           \end{enumerate}
\end{prop}

\begin{proof}
    Statement $(i)$ follows from Corollaries \ref{cor:main_chain} and \ref{cor:main_Fermat}. Similarly, $(ii)$ follows from Corollaries  \ref{cor:main_loop} and \ref{cor:main_Fermat}, and $(iii)$ follows from Corollary  \ref{cor:main_Fermat}. 
\end{proof}

In particular, we deduce:

\begin{thm}[{\cite[Conjecture 5.3]{UL}}]\label{thm:new_conj}
    Let $w$ be a suspended polynomial and let $F$ (resp. $L$) denote the Milnor fiber (resp. link) of the singularity defined by $\check{w}=0$. Consider a small $\Q$-factorialization of such singularity, and let $F_{1},\ldots,F_{g_w}$ be the Milnor fibers of the resulting $\Q$-factorial singularities. Then for any integer $k\leq 0$ we have
    \begin{equation}\label{eq:new_conj}
         \dim SH^{2k}(F)=\sum_{i=1}^{g_w}\dim SH^{2k}(F_i)+b_2(L)
    \end{equation}
      where $g_w$ is defined as in Proposition \ref{prop:ranks}.
\end{thm}

\begin{proof}
   First, we note that (\ref{eq:new_conj}) is independent of the choice of $\Q$-factorialization. Now, to obtain a small $\Q$-factorialization of the threefold singularity $\check{w}=0$ we can blow up $\Spec\left(\C[x_1,x_2,x_3,x_4]/(\check{w})\right)$ exactly $b_2(L)$ times along the ideals that correspond to the Weil divisors which are not $\Q$-Cartier (see e.g. \cite[Section 5]{wemyss2} and \cite{wemyss}). This yields a threefold with only $\Q$-factorial singularities which are all isomorphic to the singularity defined by $\tilde{w}=0$, where $\tilde{w}$ is as in (\ref{pol_qfac}). Moreover, there are exactly $g_w$ such singularities and the birational modification is small since we introduce a chain of $b_2(L)$ exceptional curves, each of which is isomorphic to $\P^1$.
   
   Formula (\ref{eq:new_conj}) then follows from Proposition \ref{prop:ranks}.
   \end{proof}

\section{Applications of the bigrading to contact topology}\label{sec:top}

In this final section of the paper, we will apply the computations from Section \ref{sec:chain}, and the formulas therein, to study the contact topology of the links of the singularities defined by $\check{w}$, whenever $w$ is of chain, loop or Fermat type. \par

More precisely, given any two invertible polynomials as in (\ref{example-chain}), (\ref{example-loop}) or (\ref{example-Fermat}), our goal is to compare, as contact manifolds, the links of the singularities determined by their mirror duals (Definition \ref{dual}). For this, we will make use of some invariants that allow us to detect whether the two links are contactomorphic or not. Of particular interest to us will be what we call the \textit{first concentrated degree} and its \textit{support}, which we will introduce in Section \ref{subsec:fcd}; and the \textit{formal period}, which we introduce in Section \ref{sec:rho1}. 

Even though the first concentrated degree and the formal period are defined in different ways, for technical reasons, they ultimately both serve the same purpose:
Given a suspended polynomial $w$, they allow us to detect the largest integer $k\leq 0$ such that one of the sets $\mathcal{Y}_k,\tilde{\mathcal{Y}}_k$ or $\mathcal{Z}^{k-1}_{-1,-1}$ (depending on the type of $w$) is non-empty. Then looking at $SH^{2k}(F)$ (or at $SH^{2k+1}(F)$) and its bigraded pieces, these numerical invariants can be used to distinguish the contact structures coming from different singularities $\check{w}=0$.

\subsection{Smooth deformations of hypersurface singularities}

In view of Gray's stability theorem, we first introduce, in a self-contained manner, an appropriate notion of smooth deformation:

\begin{defi}\label{def:deformation}
    Let $f_{i=0,1}:\C^{n+1}\rightarrow \C$ be two polynomials defining isolated hypersurface singularities at the origin $0\in \C^n$. We will call the two polynomials $f_0$ and $f_1$ \textbf{deformation equivalent} if there exists a smooth one-parameter family of polynomials $f_t:\C^n \to \C$, such that:
    \begin{enumerate}[(i)]
        \item each $f_t$ again defines an isolated hypersurface singularity at the origin,
        \item there exists $0<\varepsilon<\delta$, independent of $t$, such that the families
        \[L_t=f_t^{-1}(0)\cap S^{2n-1}(\delta)\quad \left(\text{resp.} \quad  F_t=f_t^{-1}(\varepsilon)\cap B^{2n}(\delta)\right)\] provide a smooth isotopy between the links (resp. the Milnor Fibers) of the two singularities coming from the polynomials $f_0$ and $f_1$.
    \end{enumerate}
We will call such a smooth family $f_t$ a \textbf{smooth deformation} between $f_0$ and $f_1$.
\end{defi}

This type of deformations, which are in particular $\mu$-constant (cf. \cite{LeRamanujam}), appear implicitly in \cite{EL}, for example in Remark 1.7.  And, as already mentioned, our interest in considering it lies in the following result: 

\begin{lema}\label{crit-links}
 Let $f_{i=0,1}:\C^{n}\rightarrow \C$ be two polynomials defining isolated hypersurface singularities at the origin $0\in \C^n$. If $f_0$ and $f_1$ are deformation equivalent, then the corresponding links $L_{f_0}$ and $L_{f_1}$ are contactomorphic via an ambient contactomorphism of $(S^{2n-1},\xi_{std})$. In addition, the two singularities have the same Milnor number.
\end{lema}

\begin{proof}
If $f_0$ and $f_1$ are deformation equivalent, then by assumption there exists a smooth isotopy through contact submanifolds of $(S^{2n-1},\xi_{std})$ between the contact manifolds $L_{f_0}$ and $L_{f_1}$. Such an isotopy extends to an ambient contact isotopy (see for example \cite{geiges}) thus providing the desired ambient contactomorphism.\par

The claim about the Milnor numbers follows directly from the assumption that the Milnor fibers are isotopic to each other.
\end{proof}

The following two examples illustrate how, in practice, one can apply the criterion from Lemma \ref{crit-links}:

\begin{exe}\label{exe:def0}
The prototypical example is a family of lines passing through the origin, i.e. polynomials of the form
\[
f(z_1,z_2)=\prod_{i=1}^m(\alpha_i z_1+\beta_i z_2).
\]

It's a classical observation that when $m\geq 4$ these have moduli, given by the cross-ratio, yet they are all deformation equivalent, in the sense of Definition \ref{def:deformation}.\par

Indeed, consider polynomials $f_0$ and $f_1$ as above, with parameters $(\alpha_{0,i},\beta_{0,i})$ respectively. There exists a smooth family $(\alpha_{t,i},\beta_{t,i})$ that defines a smooth deformation $f_t$ between $f_0$ and $f_1$ such that 

\begin{itemize}
    \item $(\alpha_{t,i},\beta_{t,i})\neq (0,0)$ for $1\leq i\leq m$ and $t\in [0,1]$,
    \item $(\alpha_{t,i},\beta_{t,i})\neq c\cdot(\alpha_{t,j},\beta_{t,j})$ for $1\leq i<j\leq m$ and $t\in [0,1]$
\end{itemize}
Such a family exists because the above conditions define the complement of a closed subset of $\C ^{2m}$ with real codimension at least $2$. In other words, the space of the parameters $(\alpha,\beta)$ such that defining polynomial is square-free, is path-connected.\par

Since for each $t\in [0,1]$ the polynomial $f_t$ is weighted homogeneous and $f_t^{-1}(0)$ is a union of complex lines, it is clear that the condition $(ii)$ for $f_t$ to be a smooth deformation is satisfied for any $0< \varepsilon <\delta$.
\end{exe}

\begin{exe}\label{exe:def1}
Example \ref{exe:def0} can be generalized slightly to prove that certain singularities that admit small resolutions have contactomorphic links (cf. Proposition \ref{rho2}). 

Consider, for example, the chain and loop-type polynomials:
\[
f_0=x_1^2+x_2^2+x_3\left(x_3^{\tilde{a}c}+x_4^c\right)\quad \text{and} \quad 
    f_1=x_1^2+x_2^2+x_3x_4\left(x_3^{\tilde{a}(c-1)}+x_4^{(c-1)} \right).
\]
Letting  $h_i(x_3,x_4)\coloneqq f_i(0,0,x_3,x_4)$ we have that both $h_i$ are of the form 
\[
h_i(x_3,x_4)=x_3\prod_{j=1}^{c}(\alpha_jx_3^{\tilde{a}}+\beta_jx_4).
\]
Now, as in Example \ref{exe:def0}, the $h_i$ are deformation equivalent, hence so are the $f_i$. Thus, their corresponding links are contactomorphic.\par 

On the other hand, observe that the two polynomials
\[
f_0=x_1^2+x_2^2+x_3x_4(x_3^{2}+x_4^2)\quad \text{and} \quad 
    f_1=x_1^2+x_2^2+x_3(x_3^3+x_4^{6})
\]
define singularities having small resolutions with the same number of exceptional curves, but they are not deformation equivalent -- their Milnor numbers are not the same (see Proposition \ref{hh3}). Indeed, we will see in Proposition \ref{rho2} that their links are not contactomorphic.
\end{exe}

Furthermore, we also observe that condition (ii) in Definition \ref{def:deformation} is quite strong:

\begin{nonexe}
The family of polynomials $f_t(x_1,x_2)=x_1^2+x_2^2(x_2+t)$ does \textbf{not} give a smooth deformation between $f_0$ and $f_1$ as in Definition \ref{def:deformation}. The reason is that $f_t$ has a critical point not only at $(0,0)$ but also at $(0,-2/3t)$. Therefore, the $\delta$ appearing in Definition \ref{def:deformation} (ii) cannot be chosen independently of $t$. In particular, $f_0$ and $f_1$ do not even have the same Milnor numbers.
\end{nonexe}

\subsection{Bigradings on $HH^*(\mathscr{C}_{w})$ and $SH^*(F)$}\label{subsec:bigrading}

We now recast the computations of $HH^*(\mathscr{C}_{w})$ from Section \ref{sec:chain} to $SH^*(F)$ by focusing on a key contact invariant, namely the bigrading on $SH^*$ that comes from its Gerstenhaber structure (see Theorem \ref{thm:bigrading_invariant}). 

Recall that, as explained in Section \ref{sec:back}, $HH^*(\mathscr{C}_{w})$ and $ SH^*(F)$ have a Gerstenhaber structure, hence a $\Z\times \C$ bigrading coming from the adjoint representation and the corresponding weight space decomposition.

On the other hand, in \cite[Lemma 2.7]{EL} the authors show that $HH^*(\Cw)$ is isomorphic to the Hochschild cohomology of a formal $\mathbb{Z}$-graded algebra. 
Hence, $HH^*(\Cw)$ is also equipped with a $\mathbb{Z}\times\mathbb{Z}$-bigrading, which comes from an additional $\mathbb{Z}$-grading on the Hochschild cochains (compare Example $4.2$ in \cite{EL} and Remark 5.2 in \cite{LekiliUeda}). When $HH^{1,0}$ is one-dimensional, these two bigradings are scale equivalent, meaning that they agree for a specific identification of $HH^{1,0}$ with $\mathbb{C}$.\par

To distinguish between the two bigradings, in what follows we will denote by $SH^{p,q}(F)$ a bigraded piece corresponding to the weight-space decomposition, and by $HH^{p,q}$ a bigraded piece corresponding to the $\mathbb{Z}\times \mathbb{Z}$-bigrading. Under this convention, we have that $HH^{p,q}\subset HH^{p+q}$ and $SH^{p,q}\subset SH^p$. Moreover, as we will see later (Corollary \ref{cor:hhsh}), we have that $HH^{p-q,q}\simeq SH^{p,q}$. \par

Interestingly, we can determine the rank of each piece $HH^{p-q,q}$ in terms of the total exponents $b_0$ of the monomials $\m=\x$ contributing to $HH^{p}$. More precisely,

\begin{lema}[{\cite[Section 4]{EL}}]
If a monomial $\underline{m}=x_0^{b_0}x_1^{b_1}x_2^{b_2}x_3^{b_3}x_4^{b_4}$ contributes to $HH^r(\Cw)$, then it contributes to the bigraded piece $HH^{r-3b_0,3b_0}$. 
\label{lem:b0}
\end{lema}

The computations of the previous section describe such exponents. For the convenience of the reader, we summarize the relevant data in a slightly different fashion.

\begin{prop}\label{chainb0}
Let $w$ be a chain type polynomial as in $(\ref{example-chain})$. Let $k\in \Z_{\leq 0}$ and let $m_3$ be an integer that belongs to one of the sets $\mathcal{X}_k, \mathcal{Y}_k, \mathcal{W}_k$ and $\mathcal{Z}_{i,k}$ appearing in the formula (\ref{formula-chain}) from Theorem \ref{main_general}. Then, depending on the sets, there exists a good monomial $\m=\x$ contributing to $HH^{2k-3b_0,3b_0}(\Cw)$ (as well as to $HH^{2k+1-3b_0,3b_0}(\Cw)$), where
\begin{itemize}
    \item $b_0=-a(m_3+k)$ if $m_3\in \mathcal{W}_k$, 
    \item $b_0=(b-1)m_3-k$ if $m_3 \in \mathcal{X}_k \cup \mathcal{Y}_k$
    \item $b_0=(b-1)m_3-k+i$ if $m_3 \in \mathcal{Z}_{i,k}$
\end{itemize}
Moreover, all contributions to  $HH^{2k}(\Cw)$ (and to $HH^{2k+1}(\Cw)$) arise this way.
\end{prop}

\begin{proof}
    It follows from the proof of Theorem \ref{main_general}.
\end{proof}

Similarly,

\begin{prop}\label{loopb0}
Let $w$ be a loop type polynomial as in $(\ref{example-loop})$. Let $k\in \Z_{\leq 0}$ and let $m_3$ be an integer that belongs to one of the sets $\tilde{\mathcal{X}}_k, \tilde{\mathcal{Y}}_k, \tilde{\mathcal{W}}_k$ and $\tilde{\mathcal{Z}}_{j,k}$ appearing in the formula (\ref{formula-loop}) from Theorem \ref{main_general_loop}. Then, depending on the sets, there exists a good monomial $\m=\x$ contributing to $HH^{2k-3b_0,3b_0}(\Cw)$ (as well as to $HH^{2k+1-3b_0,3b_0}(\Cw)$) where
\begin{itemize}
    \item $b_0=-(c-1)m_3-ck$ if $m_3 \in \tilde{\mathcal{W}}_k$,
    \item $b_0=(d-1)m_3-k$ if $m_3 \in \tilde{\mathcal{X}}_k \cup \tilde{\mathcal{Y}}_k$,
    \item $b_0=(d-1)m_3-k+j$ if $m_3 \in \tilde{\mathcal{Z}}_{j,k}$
\end{itemize}
Moreover, all contributions to  $HH^{2k}(\Cw)$ (and to $HH^{2k+1}(\Cw)$) arise this way.
\end{prop}

\begin{proof}
    It follows from the proof of Theorem \ref{main_general_loop}.
\end{proof}

And

\begin{prop}\label{fermatb0}
   Let $w$ be a Fermat-type polynomial as in $(\ref{example-Fermat})$. Let $k\in \Z_{\leq 0}$ and let $(m_3,-s-m_3)$ be a pair of integers that belongs to one of the sets $\F_{i,j}^s$ appearing in the formula (\ref{formula-Fermat}) from Theorem \ref{main_Fermat}. Then, depending on $s,i$ and $j$, there exists a good monomial $\m=\x$ contributing to $HH^{2k-3b_0,3b_0}(\Cw)$ (as well as to $HH^{2k+1-3b_0,3b_0}(\Cw)$) where 
   \[
   b_0=i+em_3=j-f(s+m_3).
   \]
   
   Moreover, all contributions to  $HH^{2k}(\Cw)$ (and to $HH^{2k+1}(\Cw)$) arise this way.
\end{prop}

\begin{proof}
    It follows from the proof of Theorem \ref{main_Fermat}.
\end{proof}

Now, if one wants to compute the bigrading on $SH^*$ instead, then one needs to first establish how to relate such bigrading to the (algebraic) one as in Lemma \ref{lem:b0}. The relationship is explained by Lemma \ref{lem:hh1} and Corollary \ref{cor:hhsh} below.

For the polynomials we consider here, $HH^{1,0}(\Cw)$ is always one-dimensional:

\begin{lema}\label{lem:hh1}
If $w$ is a suspended polynomial, then each non-vanishing bigraded piece in $HH^1(\Cw)$ is one-dimensional. Moreover, if $\dim HH^{1}(\mathscr{C}_w)=\ell$, then
\[HH^{1}=\bigoplus_{i=0}^{\ell-1}HH^{1-3i,3i}\]
\end{lema}
In particular, $\dim HH^{1,0}(\Cw)=1$.
\begin{proof}
When $w$ is of chain type this follows from the proof of Corollary \ref{cor:main_chain}(Case 3). Similarly, for loop-type (resp. Fermat-type) polynomials it follows from the proof of Corollary \ref{cor:main_loop} (resp. Corollary \ref{cor:main_Fermat}). 
\end{proof}

As a consequence, the bigrading on $SH^*$ coming from the weight space decomposition is, in fact, scale-equivalent to the algebraic bigrading on $HH^*$:

\begin{cor}\label{cor:hhsh}
Let $w$ be a suspended polynomial. Then, up to scale-equivalence, we may assume that each contribution to $HH^r(\Cw)$ coming from a good pair $(\gamma,\m=\x)$ corresponds to a contribution to the bigraded piece $SH^{r,b_0}(F) \subset SH^{r}(F)$, where, as before, $F$  denotes the Milnor fiber of the singularity defined by the dual polynomial $\check{w}$.
\end{cor}

Therefore, the bigrading on $SH^*(F)$ can be used to produce contact invariants of the link of the singularity defined by $\check{w}$, as first shown by Evans and Lekili. The following statement follows from \cite[Corollary 4.5]{EL}: 

\begin{thm}\label{thm:bigrading_invariant}
Let $(L_1,\xi_1)$ and $(L_2,\xi_2)$ be the links of two invertible $cA_n$ singularities, and denote by $F_1$ and $F_2$ the corresponding Milnor fibers. If $(L_1,\xi_1)$ and $(L_2,\xi_2)$ are contactomorphic, then the groups $SH^r(F_1)$ and $SH^r(F_2)$ are isomorphic for every $r\leq 1$, and they have scale-equivalent bigradings. 

In particular, if for some $r\leq 1$ the bigradings on $SH^r(F_1)$ and $SH^r(F_2)$ are not scale-equivalent, then the corresponding links cannot be contactomorphic.
\end{thm}

Note that, given Corollary \ref{cor:hhsh}, the data we need to apply Theorem \ref{thm:bigrading_invariant} is exactly the data we described in Propositions \ref{chainb0}, \ref{loopb0} and \ref{fermatb0} above. In particular, from now on, we will adopt the following convention:

\begin{conv}\label{conv:bidegree}
     We will say a bigraded piece $SH^{r,b_0}\subset SH^r(F)$ has bidegree $b_0$.
\end{conv}

With this in mind, we introduce the contact invariants that will allow us to distinguish a plethora of links of isolated $cA_n$ singularities in Section \ref{sec:comparison}. 

\subsection{Some useful contact invariants}

We first observe that given $w$, it follows from Theorem \ref{thm:bigrading_invariant} above (or \cite[Corollary 4.5]{EL}) that the minimum (resp. maximum) rank of $SH^{\leq 1}(F)$, hence the minimum (resp. maximum) rank of $HH^{\leq 1}(\Cw)$, is a contact invariant of the link $L$ of the singularity defined by $\check{w}$. 

This motivates the following definition:

\begin{defi}\label{defimin}
Let $w$ be a suspended polynomial, and let $F$ denote the Milnor fiber of the singularity defined by its dual $\check{w}$. We define:
\begin{enumerate}[(i)]
\item $\rho(\check{w}) \coloneqq \min\{\textrm{rk}\,(SH^r(F)):r\leq 1\}$
\item $\lambda(\check{w})\coloneqq \max\{\textrm{rk}\,(SH^r(F)):r\leq 1\}$
\end{enumerate}
\end{defi}

Then we have the following result:

\begin{lema}\label{lema:min-max}
Let $w_1$ and $w_2$ be two suspended polynomials, and let $(L_1,\xi_1)$ and $(L_2,\xi_2)$ be the contact links of the singularities defined by their duals $\check{w}_1$ and $\check{w}_2$. If $L_1$ and $L_2$ are contactomorphic, then $\rho(\check{w}_1)=\rho(\check{w}_2)$ and $\lambda(\check{w}_1)=\lambda(\check{w}_2)$. 
\end{lema}

For $w$ an invertible polynomial of chain, loop, or Fermat type, we have already computed the numbers $\rho(\check{w})$ and $\lambda(\check{w})$ in Corollaries \ref{cor:main_chain}, \ref{cor:main_loop} and \ref{cor:main_Fermat}. Here we additionally observe that $\rho$ is an invariant of the diffeomorphism type of the link. More precisely,

\begin{thm}\label{thm:h2L}
    If $L$ denotes the link of an invertible $cA_n$ singularity that is defined by a polynomial $\check{w}$, then $\rho(\check{w})=b_2(L)$.
\end{thm}

\begin{proof}
    Given $w$ as in (\ref{example-chain}), (\ref{example-loop}) or (\ref{example-Fermat}), the number $\rho(\check{w})+1$ is always equal to the number of branches of the curve singularity defined by $\check{w}$ restricted to the plane $x_1=x_2=0$, which is known to be precisely the rank of $H^2(L)$ plus one, see e.g. \cite{caibar2} and \cite{caibar1},  or \cite[Proposition 2.2.6]{FlipsFlops}.  
\end{proof}

\subsubsection{The first concentrated degree}\label{subsec:fcd}

One may observe from the proofs of Propositions \ref{chainb0}, \ref{loopb0} and \ref{fermatb0} (and Corollary \ref{cor:hhsh}) that, whenever $\rho\geq 2$, the contributions to $SH^{2k}(F)$ (and to $SH^{2k+1}(F)$) are always either
 spread-out, meaning they contribute to at least two different bigraded pieces; or
    they are \textit{concentrated}, meaning they contribute to a single bigraded piece. In particular, it becomes relevant to our purposes the introduction of the following invariant:

\begin{defi}\label{def:fcd}
Given a suspended polynomial $w$, as throughout the paper, let $F$ denote the Milnor fiber of the singularity defined by $\check{w}$. If $\rho(\check{w})\geq 2$, we define the \textbf{first concentrated degree} (\textbf{fcd}) of $SH^*(F)$, denoted by $\kappa(\check{w})$, to be the largest integer $k\leq 0$ such that all the contributions to $SH^{2k+1}(F)$ appear in the same bidegree, i.e. they are all concentrated in a single bigraded piece. Moreover, we will call the corresponding bidegree the \textbf{support} of the fcd, denoted by $\sigma(\check{w})$.
\end{defi}

For the polynomials we are interested in, we can compute the numbers $\kappa(\check{w})$ and $\sigma(\check{w})$ by refining our computations from Section \ref{sec:chain}, which yields the following:

\begin{prop}
Let $w$ be a suspended polynomial and assume $\rho(\check{w})\geq 2$. Then $\kappa(\check{w})$ and $\sigma(\check{w})$ are contact invariants of the link $L$ of the singularity defined by $\check{w}$. Moreover, Table \ref{table:invariants} expresses these invariants, depending on whether $w$ is of chain, loop or Fermat type:
\end{prop}

\begin{table}[h!]
\centering
\renewcommand{\arraystretch}{2.5}
\begin{tabular}{||c || c| c| c| c||} 
 \hline
  & $\lambda=\lambda(\check{w})$ & $\rho=\rho(\check{w})$ & $\k=\kappa(\check{w})$ & $\sigma=\sigma(\check{w})$   \\ [0.5ex] 
 \hline\hline
 chain & $\min\{a-1,b\}$ & $\gcd(a-1,b)$ & $1-\frac{a-1+b}{\rho}$ & $\frac{ab}{\rho}-1$ \\ 
 \hline
 loop & $\min\{c,d\}$ & $\gcd(c-1,d-1)+1$ & $1-\frac{c+d-2}{\rho-1}$ & $\frac{cd-1}{\rho-1}-1$ \\
 \hline
 Fermat & $\min\{e,f\}-1$ & $\gcd(e,f)-1$ & $1-\frac{e+f}{\rho+1}$ & $\frac{ef}{\rho+1}-1$ \\
 \hline
\end{tabular}
\caption{Invariants of the contact topology of $L$}
\label{table:invariants}
\end{table}

\begin{proof}
The first part of the statement follows from \cite[Corollary 4.5]{EL}, see also \cite[\S 2, Section 4.6]{EL}. For the second part, the computation of the invariants, we can argue as follows:
\begin{enumerate}
    \item[\bf{Case 1}] If $w$ is a chain-type (resp. loop-type) polynomial, then we observe that for the integer $k=\k$ as in Table \ref{table:invariants} we have that $\mathcal{X}_k=\mathcal{Y}_{k}=\{(a-1)/\rho\}$ (resp. $\tilde{\mathcal{X}}_{k}=\tilde{\mathcal{Y}}_{k}=\{(c-1)/\rho\}$ ) and all the other sets appearing in formula (\ref{formula-chain}) (resp. in formula (\ref{formula-loop})) are empty. This, combined with Corollary \ref{chainb0} (resp. Corollary \ref{loopb0}), tells us  the contributions to $SH^{2k+1}$ are concentrated in bidegree $\sigma$, where $\sigma$ is as in Table \ref{table:invariants}.\par

Moreover, when $\rho \geq 2$ we can check that $2\k+1$ is indeed the \textbf{fcd}. This follows from the computations in the proof of Corollary \ref{cor:main_chain} (resp. Corollary \ref{cor:main_loop}), since for any integer $k$ larger than $\kappa$, we have that at least one of the sets $\mathcal{Z}_{i,k}$  (resp. $\tilde{\mathcal{Z}}_{j,k}$) is non-empty  and  we also have that $\mathcal{X}_k \cup \mathcal{W}_k\neq \emptyset$ (resp. $\tilde{\mathcal{X}}_k \cup \tilde{\mathcal{W}}_k\neq \emptyset$). Furthermore, by Proposition \ref{chainb0} (resp. Proposition \ref{loopb0}) and Corollary \ref{cor:hhsh} the corresponding contributions are not to the same bigraded piece.\par

    \item[\bf{Case 2}] Similarly, if $w$ is a polynomial of Fermat type, then for the integer $k=\k$ as in Table \ref{table:invariants} we have that $\F^{k-1}_{-1,-1}=\left\{\left(\frac{f}{\rho+1},\frac{e}{\rho+1}\right)\right\}$ and all the other sets appearing in formula (\ref{formula-Fermat}) are empty. This, combined with Corollary \ref{fermatb0}, tells us  the contributions to $SH^{2k+1}$ are concentrated in bidegree $\sigma$, where $\sigma$ is as in Table \ref{table:invariants}.\par
 The fact that $2\k+1$ is indeed the \textbf{fcd} when $\rho\geq 2$ then follows from Theorem \ref{main_Fermat}. For any integer $k\leq 0$ larger than $\kappa$, we have that at least $\gcd(e,f)-1$ among the sets $\F^k_{i,j}$ are non-empty. Furthermore, by Proposition \ref{fermatb0} and Corollary \ref{cor:hhsh}, each corresponding contribution is to a different bigraded piece.\par
\end{enumerate}
\end{proof}

\begin{rmk}
    We observe that when $\rho(\check{w})\leq 1$, then the numbers $\kappa$ and $\sigma$ from Table \ref{table:invariants} still make sense -- simply their interpretation changes. The number $\kappa$ is then the largest integer $k\leq 0$ such that either: $\mathcal{Y}_k\neq \emptyset$ (resp. $\tilde{\mathcal{Y}}_k\neq \emptyset$) if $w$ is of chain type (resp. loop type); or $\F^{k-1}_{-1,-1}\neq \emptyset$ if $w$ is of Fermat type. The number $\sigma$ is still the associated bidegree.
\end{rmk}

\begin{rmk}\label{rmk:lct}
    It is also intriguing to observe that, for all three types of polynomials and independent of $\rho$, the rational number $\frac{1-\kappa}{\sigma+1}$ is precisely the log canonical threshold of the plane curve singularity defined by $\check{w}$ restricted to $x_1=x_2=0$. This can be seen using the formulas in \cite[Theorem 1.2]{Kuwata}, for example. The reader is also referred to \cite[Section 8]{pairs} for the details on this invariant of singularities.
\end{rmk}

\subsection{Comparing different links when $\rho\geq 2$}\label{sec:comparison}

We now have all the tools we need to compare the links of any two invertible $cA_n$ singularities $\check{w}=0$ satisfying $\rho=\rho(\check{w})\geq 2$.  

We begin by comparing the links associated to any two polynomials lying in the same family. More precisely, we first prove Propositions \ref{fvsf}, \ref{cvsc} and \ref{lvsl} below:

\begin{prop}\label{fvsf}\emph{(Fermat vs Fermat)}
    Let $w_1=\check{w}_1=x_1^2+x_2^2+x_3^e+x_4^f$ and let  
     $w_2=\check{w}_2=x_1^2+x_2^2+x_3^{e'}+x_4^{f'}$. If $\rho(\check{w}_1)\geq 2 $ and $\rho(\check{w}_2)\geq 2$, the links of the singularities defined by $\check{w}_1$ and $\check{w}_2$ are contactomorphic if and only if, up to swapping the variables $x_3$ and $x_4$, we have $\check{w}_1=\check{w}_2$.
\end{prop}

\begin{proof}
Let $w_1=\check{w}_1$ and $w_2=\check{w}_2$ be two Fermat-type polynomials as in the statement. Let $g\coloneqq\gcd(e,f)$ and let $g'\coloneqq \gcd(e',f')$. Then, up to swapping the variables $x_3$ and $x_4$, we can assume that $e\leq f$ and $e'\leq f'$. 

Now, if the links of the singularities defined by $w_1$ and $w_2$ are contactomorphic, then all the invariants from Table \ref{table:invariants} must agree. Therefore, we must have $e=e'$ and $g=g'$, since $\rho(w_1)=\rho(w_2)$ and $\lambda(w_1)=\lambda(w_2)$.\par

Finally, since $\rho\geq 2$ (hence  $g=g'\geq 3$), by further looking at the invariants $\kappa$ and $\sigma$, as in Table \ref{table:invariants}, we conclude that $f=f'$ as well.\par 
\end{proof}

Similarly,

\begin{prop}\label{cvsc}\emph{(chain vs chain)}
 Let $w_1=x_1^2+x_2^2+x_3^ax_4+x_4^b$ be a chain-type polynomial such that $\rho(\check{w}_1)\geq 2 $, and let      $w_2=x_1^2+x_2^2+x_3^{a'}x_4+x_4^{b'}$ be another chain-type polynomial with $\rho(\check{w}_2)\geq 2$. The links of the singularities defined by $\check{w}_1$ and $\check{w}_2$ are contactomorphic if and only if  $\check{w}_1 = \check{w}_2$.
\end{prop}

\begin{proof}
 We can argue as in the proof of Proposition \ref{fvsf}. If the singularities defined by $\check{w}_1$ and $\check{w}_2$ have contactomorphic links, then the two polynomials must have the same invariants $\rho,\lambda, \k$ and $\sigma$, as in Table \ref{table:invariants}, which readily implies that they are equal. We must have $a=a'$ and $b=b'$.
\end{proof}

And, lastly,

\begin{prop}\label{lvsl}\emph{(loop vs loop)}
Let $w_1=x_1^2+x_2^2+x_3^cx_4+x_3x_4^d$  be a  polynomial of loop-type such that $\rho(\check{w}_1)\geq 2 $, and let  
     $w_2=x_1^2+x_2^2+x_3^{c'}x_4+x_3x_4^{d'}$ be another loop-typ polynomial with $\rho(\check{w}_2)\geq 2$. The links of the singularities defined by $\check{w}_1$ and $\check{w}_2$ are contactomorphic if and only if, up to swapping the variables $x_3$ and $x_4$,  we have $\check{w}_1=\check{w}_2$.
\end{prop}

\begin{proof}
   Once more, if the links of the singularities defined by $\check{w}_1$ and $\check{w}_2$ are contactomorphic, then all the corresponding invariants as in Table \ref{table:invariants} must agree, which then implies that either  $c=c'$ and $d=d'$; or $c=d'$ and $d=c'$.
\end{proof}

We now compare the links of two singularities coming from different types of invertible polynomials by proving:

\begin{prop}\label{rho2}
Let $w_1$ and $w_2$ be two invertible polynomials of different types among loop, chain and Fermat. Suppose both $\rho(\check{w}_1)\geq 2$ and $\rho(\check{w}_2)\geq 2$. The singularities defined by $\check{w}_1$ and $\check{w}_2$ have contactomorphic links if and only if the polynomials $\check{w}_1$ and $\check{w}_2$  are deformation equivalent (as in Definition \ref{def:deformation}). In particular, this can only happen if both singularities admit a small resolution with the same number of exceptional curves.
\end{prop}

\begin{proof}
Let us first consider the case when $w_1=w^{a,b}_{chain}=x_1^2+x_2^2+x_3^ax_4+x_4^b$ and  $w_2=w^{c,d}_{loop}=x_1^2+x_2^2+x_3^cx_4+x_3x_4^d$. 

If the singularities defined by $\check{w}_1$  and $\check{w}_2$ have contactomorphic links, then the numbers $\kappa_1\coloneqq \kappa(\check{w}_1)$ and $\kappa_2 \coloneqq\kappa(\check{w}_2)$ are equal. Moreover, the same is true for the numbers $\sigma_1\coloneqq \sigma(\check{w}_1)$ and $\sigma_2 \coloneqq\sigma(\check{w}_2)$. In addition, if $\rho_1\coloneqq \rho(\check{w}_1)=\gcd(a-1,b)$ and $\rho_2 \coloneqq \rho(\check{w}_2)=\gcd(c-1,d-1)+1$, then we further know $\rho_1=\rho_2$.

In particular, after some rewriting, the equations $\kappa_1=\kappa_2$  and $\sigma_1=\sigma_2$ become:
\[
\frac{1-a-b}{\rho_1}=\frac{2-c-d}{\rho_1-1} \quad \text{and} \quad   \frac{ab}{\rho_1}=\frac{cd-1}{\rho_1-1}  
\]

Now, we also know that the two numbers $\lambda_1\coloneqq \lambda(\check{w}_1)=\min\{a-1,b\}$ and $\lambda_2\coloneqq \lambda(\check{w}_2)=\min\{c,d\}$ agree as well, and that we can always assume the latter to be equal to $c$. Therefore, in the two equations above,  it will suffice to consider two possibilities, namely that either $c=a-1$ or $c=b$. 

When $c=b$, the above equations give us $b=\rho_1$, hence $c-1=\rho_1-1$ and $\rho_1(d-1)=(a-1)(\rho_1-1)$. Thus, in this case the polynomials $\check{w}_1$ and $\check{w}_2$ are of the form: 
\begin{align*}
\check{w}_1&=x_1^2+x_2^2+x_3(x_3^{\tilde{a}c}+x_4^{c})\\
\check{w}_2&=x_1^2+x_2^2+x_3x_4(x_3^{c-1}+x_4^{\tilde{a}(c-1)})
\end{align*}

where $\tilde{a}\coloneqq(a-1)/\rho_1$.  Note that, after swapping $x_3$ with $x_4$, indeed they are deformation equivalent, as explained in Example \ref{exe:def1}. In addition, it is also true that $b=\gcd(a-1,b)$ and $d-1=\gcd(c-1,d-1)$.

Next, if $c=a-1$, then the equations $\kappa_1=\kappa_2$  and $\sigma_1=\sigma_2$  yield 
\[
b=c^2-\frac{\rho_1}{\rho_1-1}(c-1)^2 \quad \textrm{and} \quad d=\frac{\rho_1-1}{\rho_1}(c^2+c)-c^2+c+1.
\]

But then the condition $a-1\leq b$ implies that $c=a-1=\rho_1=d=b$, which brings us back to the previous case (that is, $c=b$).

Finally, it is routine to check that if we now consider the other two remaining possibilities for $w_1$ and $w_2$ the same type of reasoning will apply. That is, if we assume $w_1$ is either of chain type or of loop type, and $w_2$ is of Fermat type; then, by equating all their corresponding invariants, we do obtain the necessary constraints on the exponents of $\check{w}_1$ and $\check{w}_2$ that allow us to conclude Proposition \ref{rho2} holds. 
\end{proof}

All possible comparisons between suspended polynomials with $\rho\geq 2$ have now been dealt with. It only remains to discuss the case $\rho\leq 1$.

\subsection{The remaining case: $\rho\leq 1$}\label{sec:rho1}

We would like to now compare the links of singularities defined by suspended polynomials $\check{w}$ such that $\rho(\check{w})\leq 1$. In this case, the contributions to $SH^{2\k}$ and $SH^{2\k+1}$, with $\kappa$ as in Table \ref{table:invariants}, are necessarily concentrated (if there are any). Therefore, the definition of the \textbf{fcd} does not directly apply to these singularities and we need to consider another invariant, which can be extracted from the symplectic cohomology of their Milnor fiber, and which we call the \textbf{formal period}. 

To define this invariant, we start by introducing the following auxiliary definition:

\begin{defi}
    Given a suspended polynomial $w$, and an integer $r\leq 0$, we define
    \[
    \rho_b^r(\check{w})\coloneqq \min\{s\,:\, \textrm{rk}(SH^{r,s}(F))>0\},
    \]
    where $F$ denotes the Milnor fiber of the singularity $\check{w}=0$.
\end{defi}

Then we can define the formal period as follows:

\begin{defi}
Let $w$ be a suspended polynomial, and suppose there exists an even integer $T\geq 0$ such that:
    \begin{enumerate}[(i)]
        \item $\rho(\check{w})=1= \textrm{rk}(SH^{-T}(F))$,
        \item $\lambda(\check{w})= \textrm{rk}(SH^{-(T+2)})(F))$  and
        \item $\rho_b^{-T}(\check{w})+1=\rho_b^{-(T+1)}(\check{w})$. 
\end{enumerate}
We define the \textbf{formal period} of $SH^{\leq 1}(F)$, denoted by $\theta(\check{w})$, to be the smallest $T$ with the above properties. Moreover, we write 
\[
    \rho_b(\check{w})\coloneqq \rho_b^{-\theta(\check{w})}(\check{w})
    \]
\end{defi}

\begin{rmk}
    It is a consequence of our computations in Section \ref{sec:chain} that whenever $\rho(\check{w})=1$, the number $\theta(\check{w})$ is indeed well-defined. We prove this in Lemmas \ref{fperiodfermat} and \ref{fperiodchain} below.
\end{rmk}

Let us now explain why the formal period $\theta(\check{w})$ is well-defined for a Fermat or chain-type polynomial $w$ with $\rho(\check{w})=1$, and how to explicitly compute it.

As mentioned before, the number $\theta(\check{w})$ detects the largest $k\leq 0$ for which the special sets $\mathcal{Y}_k$ or $\F^{k-1}_{-1,-1}$ from Theorems \ref{main_general} and \ref{main_Fermat} are non-empty, depending on whether $w$ is of chain or Fermat type, respectively. 

We have:

\begin{lema}\label{fperiodfermat}
Let $w=w^{e,f}_{Fermat}$ be a Fermat-type polynomial 
 
with $e \leq f$ such that $\gcd(e,f)=2$. Then the formal period is:
\[
\theta(\check{w})=\begin{cases}
    f-2  & \text{if $e=2$}\\
     e+f-2 & \text{otherwise}
\end{cases}
\]
Moreover, in the first case $SH^{-\theta}(F)$ (which is one-dimensional) is supported in bidegree $f-2$; and, in the second, in bidegree $ef/2-1$.
\end{lema}

\begin{proof}
    Let $\theta=\theta(\check{w})$ be as in the statement and let  $k=-\theta/2$. We will show $k$ is indeed the largest non-positive integer satisfying:
    \begin{enumerate}[(i)]
        \item $\textrm{rk}(SH^{2k}(F))=\rho(\check{w})=1$,
        \item$ \textrm{rk}(SH^{2(k-1)}(F))=\lambda(\check{w})=e-1$ and
        \item $\rho^{2k}_b(\check{w})+1=\rho^{2(k-1)}_b(\check{w})=\rho^{2k-1}_b(\check{w})$
    \end{enumerate}

    First, if $e=2$, then $f=2\delta$ for some $\delta\geq 1$ and $\min\{e,f\}=\gcd\{e,f\}$. Thus, it follows from Corollary \ref{cor:main_Fermat} (or Corollary \ref{conjtrue-Fermat}) that $SH^{\leq 1}(F)$ has constant rank and is equal to one. In particular, $(i)$ and $(ii)$ hold. Now, equality in (iii) follows from Corollary \ref{fermatb0}. We can check that $\rho^{2k}_b(\check{w})=f-2$ since $\F^k_{0,f-2}=\{(\delta-1,0)\}$, and we can further check $\rho^{2(k-1)}_b(\check{w})=f-1$ since $\F^{k-2}_{-1,-1}=\{(\delta,1)\}$.
    
    Next, when $e\geq 3$, then (i) and (iii) hold because $\mathcal{F}^{k-1}_{-1,-1}=\{(f/2,e/2)\}$ and Corollary \ref{fermatb0} tells us $\rho^{2(k-1)}_b(\check{w})=ef/2=\sigma(\check{w})+1=\rho^{2k}_b(\check{w})+1$. Moreover, the fact (ii) holds follows from Corollary \ref{cor:main_Fermat} since for $k=-\theta/2$ we have that $e(1-(k-1))$ is congruent to $e$ modulo $(e+f)$.

    Finally, to show the maximality of $k$ we can argue as follows: If $k=0$ there is nothing to prove.   Otherwise, let $\tilde{k}=k+n$ for some integer $1\leq n < -k$. Then either:
     $e=2$ and it follows from Corollaries \ref{cor:main_Fermat} and \ref{fermatb0} that if we replace $k$ by $\tilde{k}$, then (iii) does not hold; or
         $e \neq 2$ and Corollary \ref{cor:main_Fermat} tells us that replacing $k$ by $\tilde{k}$, then (i) will no longer hold. 
    
\end{proof}

The corresponding statement for chain-type polynomials is:

\begin{lema}\label{fperiodchain}
Let $w=w^{a,b}_{chain}$ be a chain-type polynomial with $a$ and $b$
such that $\gcd(a-1,b)=1$. Then the formal period is:
\[
\theta(\check{w})=\begin{cases}
    2(b-1)  & \text{if $a=2$}\\
     2(a+b) & \text{otherwise}
\end{cases}
\]
Moreover, in the first case $SH^{-\theta}(F)$ (which is one-dimensional) is supported in bidegree $b-1$; and, in the second, in bidegree $ab-1$.
\end{lema}

\begin{proof}
Arguing as in the proof of Lemma \ref{fperiodfermat}, the statement follows from Corollaries \ref{cor:main_chain} and \ref{chainb0}. 
\end{proof}

We can now make use of the formal period to complete our comparison of different links of invertible $cA_n$ singularities satisfying $\rho(\check{w})\leq 1$. Propositions \ref{fvsf-1} and \ref{cvsc-1} below should be regarded as the completion of Propositions \ref{fvsf} and \ref{cvsc}, respectively.

\begin{prop}\label{fvsf-1}\emph{(Fermat vs Fermat)}
     Let $w_1=\check{w}_1=x_1^2+x_2^2+x_3^e+x_4^f$ and 
     $w_2=\check{w}_2=x_1^2+x_2^2+x_3^{e'}+x_4^{f'}$ be Fermat-type polynomials such that $\rho(\check{w}_1)\leq 1$ and $\rho(\check{w}_2)\leq 1$. Then, the conclusion of Proposition \ref{fvsf} still holds, that
     is, the links of the corresponding singularities are contactomorphic if and only if, up to swapping the variables $x_3$ and $x_4$, we have $\check{w}_1=\check{w}_2$.
\end{prop}

\begin{proof}
    Assume the links of the singularities defined by $\check{w}_1=w_1$ and $\check{w}_2=w_2$ are contactomorphic. Then $\lambda(w_1)=\lambda(w_2)$ and, up to swapping the variables $x_3$ and $x_4$, we may assume $e=e'$. Let now $g$ and $g'$ be defined as in the proof of Proposition \ref{fvsf}. Since we are considering the case $\rho\leq 1$, it follows that $g\leq 2$.

  The case $g=2$ corresponds to $\rho=1$ and, in this case, the formal period is well-defined and we computed it in Lemma \ref{fperiodfermat}. By comparing $\theta(\check{w}_1)$ with $\theta(\check{w}_2)$, we immediately deduce that $f=f'$, hence the two polynomials are the same.

When $g=1$ (which corresponds to $\rho=0$), we need to argue differently. We know that  for $k\coloneqq -(e+f)$ we have $\dim SH^{2k}(F_{w_1})=e-1$ and, moreover,  $\mathcal{F}_{i,j}^k=\{f\}$ for $0\leq i\leq e-2$. In particular, it follows from Proposition \ref{fermatb0} that $SH^{2k}(F_{w_1})$ is supported  in bidegrees $(2k,\ell)$ with  $\ell$ ranging from $ef$ to $ef+e-2$. Moreover, because $SH^{2k}(F_{w_2})\simeq SH^{2k}(F_{w_1})$, the latter implies (see e.g. proof of Corollary \ref{cor:main_Fermat}) there must be some non-empty set $\mathcal{F}_{i,j}^k=\{(m_3,-k-m_3)\}$ contributing to $SH^{2k}(F_{w_2})$ in bidegree $ef$. But then this further implies that $m_3=f, i=0$ and $j=0$. Thus, we have
\[i+m_3e=j-(k+m_3)f'\Rightarrow fe=ef'\Rightarrow f=f'\]
Thus, also in this case the two polynomials coincide.
\end{proof}

Similarly, when comparing two polynomials of chain type we obtain:

\begin{prop}\label{cvsc-1}
    Let $w_1=x_1^2+x_2^2+x_3^ax_4+x_4^b$ and let      $w_2=x_1^2+x_2^2+x_3^{a'}x_4+x_4^{b'}$. If $\rho(\check{w}_1)=\rho(\check{w}_2)=1$, then
    the conclusion of Proposition \ref{cvsc} still holds: the singularities defined by $\check{w}_1$ and $\check{w}_2$ have contactomorphic links if and only if  $\check{w}_1 = \check{w}_2$.
\end{prop}

\begin{proof}
If $\check{w}_1$ and $\check{w}_2$ define singularities with contactomorphic links, then we must have that $\theta(\check{w}_1)=\theta(\check{w}_2)$ and $\rho_b(\check{w}_1)=\rho_b(\check{w}_2)$, and the result follows from Lemma \ref{fperiodchain}.
\end{proof}

Given Lemma \ref{lema:min-max} and Theorem \ref{thm:h2L}, to complete our classification of contact links, we are left with one case to consider. We need to compare the link of a singularity coming from the dual of a chain-type polynomial 
\begin{equation}\label{chain-rho1}
    w_1=x_1^2+x_2^2+x_3^ax_4+x_4^b \qquad \text{with $\gcd(a-1,b)=1$}
\end{equation}
 with the link of a singularity coming from a Fermat-type polynomial 
 \begin{equation}\label{Fermat-rho1}
     w_2=x_1^2+x_2^2+x_3^e+x_4^f \qquad \text{with $\gcd(e,f)=2$}.
 \end{equation}

This last comparison is carried out in Proposition \ref{cvsf-rho1} below, which relies on the computations of the formal period from Lemmas \ref{fperiodfermat} and \ref{fperiodchain}.

\begin{prop}\label{cvsf-rho1}
Let $w_1$ and $w_2$ be polynomials as in (\ref{chain-rho1}) and (\ref{Fermat-rho1}), respectively. Then the singularities defined by $\check{w}_1$ and $\check{w}_2$ have contactomorphic links if and only if the two polynomials $\check{w}_1$ and $\check{w}_2$ are deformation equivalent (as in Definition \ref{def:deformation}). 
\end{prop}

\begin{proof}
If $\check{w}_1$ and $\check{w}_2$ are deformation equivalent, then it follows from Lemma \ref{crit-links} and Gray's stability that the corresponding links are contactomorphic.

For the opposite implication, assume $\check{w}_1$ and $\check{w}_2$ define singularities with contactomorphic links. 

First of all, we claim that the two singularities must admit a small resolution.
Suppose they do not: then both $a$ and $e$ are not equal to $2$, and $a-1\neq b$. If all the contact invariants of the two polynomials were equal, we would reach a contradiction by arguing as follows.
If $a-1 < b$, then  $\rho_b(\check{w}_1)=\rho_b(\check{w}_2)$ and $\lambda(\check{w}_1)=\lambda(\check{w}_2)$ imply $a=e$ and  $2b=f$. But then $\theta(\check{w}_1)=\theta(\check{w}_2)$ would further tell us $a-1=b+1$, contradicting that $a-1 <  b$. Similarly, when $b < a-1$, then the same three inequalities (of the invariants $\lambda, \rho_b$ and $\theta$) would give us  $b=e-1$, $f/2=ab/(b+1)$ and $a=(b+1)(3-b)/2$. And the latter contradicts $2\leq a$ and $2\leq b$. 

Having established that the two singularities admit a small resolution, it follows from Lemma \ref{gcd-chain} that $a-1\leq b$ and $a=e=2$. And since $\theta(\check{w}_1)=\theta(\check{w}_2)$ we further have $f=2b$. 

Therefore, $\check{w}_1=x_1^2+x_2^2+x_3(x_3+x_4^b)$ and $\check{w}_2=x_1^2+x_2^2+(x_3+ix_4^b)(x_3-ix_4^b)$;  and, at last, we conclude $\check{w}_1$ and $\check{w}_2$ are indeed deformation equivalent (as in Definition \ref{def:deformation}), see e.g. Example \ref{exe:def0}. 
\end{proof}

Combining Propositions \ref{fvsf} through \ref{cvsc-1} and Proposition \ref{cvsf-rho1} we have now proved Theorem \ref{thmE}. And, as a consequence, we further obtain:

\begin{cor}\label{invmilnornumber}
If two invertible $cA_n$ singularities have contactomorphic links, then their Milnor numbers are equal.
\end{cor}

In practice, Corollary \ref{invmilnornumber} gives a useful criterion for determining when two invertible $cA_n$ singularities have contactomorphic links or not. For example, any two singularities coming from polynomials of the form
\[
\check{w}_1=x_1^2+x_2^2+x_3^e+x_4^f
\quad \text{and} \quad \check{w}_2=x_1^2+x_2^2+x_3^e+x_4^{f'}\]
will never have contactomorphic links unless $f=f'$. In fact, their correpsonding Milnor numbers are $(e-1)(f-1)$ and $(e-1)(f'-1)$, respectively.

\appendix

\section{Index positivity}
\label{app:indexpos}

In this appendix, we want to highlight the implications of the assumption of index positivity for the computations of the symplectic cohomology groups. 
Recall that we have defined (Section \ref{sec:sympcoh}) the symplectic cohomology of a Liouville domain $W$ with boundary $\Sigma$ as the Floer cohomology of its completion $\hat{W}$.

For a suitable choice of Hamiltonian function $H$ on $\hat{W}$, we can distinguish two types of periodic orbits of the corresponding Hamiltonian vector field $X_H$:
\begin{itemize}
    \item critical points of $H$ in $W$ (i.e., constant periodic orbits);
    \item $1$-periodic orbits on the level sets $\Sigma\times \{r\}$, which correspond to periodic Reeb orbits on $\Sigma$ (of period depending on $r$).
\end{itemize}

Symplectic cohomology $SH^*(W)$ is the cohomology of the complex $SC^*$ generated by all periodic orbits, with the differential given by counting Floer trajectories connecting the different orbits. We will also consider the cohomology of the subcomplex $SC^*_-$ generated by the constant periodic orbits (or the Morse subcomplex of critical points), which is called \emph{negative symplectic cohomology}. Finally, \emph{positive symplectic cohomology} is the cohomology of the quotient complex $SC^*_+=SC^*/SC^*_-$.
Not surprisingly, $SH^*_-(W)$ turns out to be isomorphic to $H^*(W)$, hence we get a (tautological) long exact sequence in cohomology:
\begin{equation}\label{les}
    \ldots\rightarrow H^*(W)\rightarrow SH^*(W)\rightarrow SH_+^*(W)\rightarrow H^{*+1}(W)\rightarrow \ldots
\end{equation}

In particular, for negative degrees $*<0$, singular cohomology vanishes, hence
\[
SH^*(W)\rightarrow SH_{+}^*(W)
\]
becomes an isomorphism, that is, symplectic cohomology and positive symplectic cohomology coincide in a negative degree. This construction is known more generally for hypersurfaces of contact type in exact symplectic manifolds (see \cite{viterbo}, \cite{CFO}).

\begin{rmk}
To have a well-defined grading, one needs to assume the following two things: $c_1(W)=0$, and the closed Reeb orbits of $\Sigma$ are contractible in $\Sigma$. 
    \end{rmk}

Since a given contact manifold can have different symplectic fillings, it is natural to ask to what extent (positive) symplectic cohomology depends on $W$.

 Under an additional assumption on the indices of the Reeb orbits, the positive symplectic cohomology can be defined by counting Floer trajectories in the positive part of the symplectization of $\Sigma$ (instead of a filling) and is thus a \emph{contact invariant}.
More precisely, the following result is proved by Uebele (\cite{Uebele}, Lemma 3.7) for $\bigvee$-shaped symplectic homology. It is an adaptation of the corresponding result for Rabinowitz Floer homology stated and proved in \cite{CFO}, Lemma 1.14, and Uebele remarks that the same statement holds for positive symplectic cohomology.
\begin{lema}
Suppose $\Sigma$ is a simply connected contact manifold satisfying the following conditions:
\begin{itemize}
    \item [(i)] $c_1(\Sigma)=0$;
    \item[(ii)] $\mu_{CZ}(\gamma)>3-n$ for all Reeb orbits $\gamma$;
    \item[(iii)] $\Sigma$ admits a Liouville filling $Z$ with $c_1(W)=0$.
\end{itemize}
We call a manifold \emph{index positive} if it  satisfies these conditions. Then $SH^*_+(W)$ is independent of the choice of $W$ and only depends on the contact boundary $\Sigma$. 
\end{lema}

\begin{rmk}
\leavevmode
\begin{itemize}
    \item[(i)] Condition $(i)$ may be replaced by $c_1(\Sigma)|_{\pi_2(\Sigma)}=0$.
    \item[(ii)] For the link of a cDV singularity, the above conditions are satisfied. Index positivity is guaranteed by McLean's theorem on index positivity of \emph{terminal} $\mathbb{Q}$-Gorenstein singularities (\cite{McLean}), and the result by Miles Reid already mentioned in the introduction, namely that the Gorenstein terminal threefold singularities are precisely the isolated cDV singularities. 
    \end{itemize}
\end{rmk}

\begin{cor}
If the contact manifold $\Sigma=\partial W$ is index positive, then the symplectic cohomology $SH^*(W)$ in negative degree $*<0$ is a contact invariant of $\Sigma$ and does not depend on the choice 
of $W$.
\end{cor}

In \cite{EL}, the above corollary is applied to distinguish contact structures on the link of a singularity using the symplectic cohomology (in negative degrees) of the corresponding Milnor fiber. 

Recall that for the negative symplectic cohomology, we have 
\[
SH^*_-(W)\cong H^{*}(W), 
\]
where $n=1/2\dim W$. For the type of singularities we consider in this manuscript, the singular cohomology is well understood: the Milnor fiber is, in fact, a smooth manifold of real dimension $6$, which has the homotopy type if a finite CW-complex of dimension $3$, namely a bouquet of $\mu$ spheres, where $\mu$ is the Milnor number of the singularity. Hence its singular cohomology vanishes above degree $3$ and in degree $3$ it has rank $\mu$.

    Index positivity has another interesting consequence: in the case where $n=1/2\dim W=3$ (\emph{hypersurface singularities}), it implies that all Reeb orbits have Conley-Zehnder index $k>0$. With the conventions in \cite{EL}, an orbit with index $k$ is a generator for the symplectic cohomology in degree $n-k=3-k<3$.
    
    In fact in the examples we consider in this paper,  something more is true.

    \begin{lema}\label{lem:deg2}
        For any threefold terminal singularity of index $1$ (hence for all the singularities we consider in this manuscript) the (positive) symplectic cohomology in degree $2$ vanishes. 
    \end{lema}

   \begin{proof} We will show that, for an appropriate choice of contact form, all Reeb orbits have Conley-Zehnder index $k>1$. This follows from the fact that Shokurov's conjecture holds in dimension $3$, and hence any threefold singularity which is terminal and of index one has minimal discrepancy $1$ (see \cite{min-discr}). Thus, by McLean's theorem \cite{McLean}, the highest minimal index is $2$, and hence the Conley-Zehnder index is at least $2$ for all periodic orbits of the Reeb flow corresponding to the contact form realizing the hmi.
    In particular, this implies that $SH^2(W)$ (and hence also $SH^2_+(W)$), which is generated by Reeb orbits of index $1$, vanishes.
    \end{proof}

    If we combine the above lemma with the long exact sequence relating positive/negative symplectic cohomology:
    \begin{equation}\label{les23}
    \ldots \rightarrow H^2(W)\rightarrow SH^2(W)\rightarrow SH_+^2(W)\rightarrow H^{3}(W)\rightarrow SH^3(W)\rightarrow SH_+^3(W)\rightarrow \ldots
\end{equation}
    we get the following result about the degree $3$ symplectic cohomology. 
    
    \begin{cor}\label{cor:milnumb}
    $SH^3(W)\cong H^3(W)\cong\mathbb{C}^\mu$, where $\mu$ is the Milnor number of the singularity.
    \end{cor}

    \section{A few explicit calculations}\label{app:expcomp}

    In this appendix, we present a few concrete examples that illustrate the applicability of the formula we obtained in Theorem \ref{main_general_loop}. In particular, we explain how to recover \cite[Theorem 3.13]{EL}.

If $w=w^{c,d}_{\text{loop}}=x_1^2+x_2^2+x_3^cx_4+x_3x_4^d$ is a loop-type polynomial as in (\ref{example-loop}), then we first observe the following holds, concerning the cardinalities of the sets $\tilde{\mathcal{W}}_k, \tilde{\mathcal{X}}_k$ and $\tilde{\mathcal{Z}}_{j,k}$ appearing in Theorem \ref{main_general_loop}:


\begin{prop}\label{prop:formula_loop}
    Given any integer $k\leq 0$, let $0\leq q<c+d-2$ be such that $(c-1)(1-k) \equiv q \mod (c+d-2)$. Then we have that:

\begin{enumerate}[(i)]
     \item  $|\tilde{\mathcal{X}}_k|= 1+\left \lfloor{\frac{(c-1)-q}{c+d-2}}\right \rfloor$
   \item[]
         \item 
         $|\tilde{\mathcal{W}}_k|= \left \lfloor{\frac{q-c}{c+d-2}}\right \rfloor-\left \lceil{\frac{q-(d-1)}{c+d-2}}\right \rceil +2 $
 \item[]
\item $|\tilde{\mathcal{Z}}_{j,k}|=\left \lfloor{\frac{q-1-j}{c+d-2}}\right \rfloor-\left \lceil{\frac{q-(c-2)-j}{c+d-2}}\right \rceil+1\\$     
\end{enumerate}

In particular, we always have that $|\mathcal{\tilde{X}}_k|+|\mathcal{\tilde{W}}_k|\geq 1$. 
  \end{prop}

\begin{proof}
   The formulas follow from the definitions of the sets $\tilde{\mathcal{W}}_k, \tilde{\mathcal{X}}_k$ and $\tilde{\mathcal{Z}}_{j,k}$.

    \end{proof}

Using Proposition \ref{prop:formula_loop} (or Corollary \ref{cor:main_loop}), we can then quickly apply formula (\ref{formula-loop}) from Theorem \ref{main_general_loop} to a few explicit choices of values of $c$ and $d$ in order to further obtain:

\begin{exe}
         When $(c,d)=(3,4)$, $r\leq 1$ and $r \equiv q \mod 10$ we have hat $\dim HH^r(\mathscr{C}_w)$ is given by:
        
\begin{center}
    \begin{tabular}{*{11}{|c}|}
        \hline $q$ & 0 & 1 & 2 & 3 & 4 & 5 & 6 & 7 & 8 & 9    \\\hline 
        $\dim HH^r(\mathscr{C}_w)$ & 3 & 3 & 2 & 2 & 3 & 3 & 2 & 2 & 2 & 2   \\ \hline
           \end{tabular}
\end{center}
         
    \end{exe}

\begin{exe}
   When $(c,d)=(5,3)$ we have that  $\dim HH^{r}(\mathscr{C}_{w})=3$ for any $r\leq 1$. Note that in this case $\tilde{\eta}=1$.
    \end{exe}

\begin{exe}
   When $(c,d)=(7,4)$ we have that  $\dim HH^{r}(\mathscr{C}_{w})=4$ for any $r\leq 1$. Note that in this case $\tilde{\eta}=2$.
    \end{exe}

\begin{exe}
    If $c=\ell$ and $d=\delta(\ell-1)+1$, for some integers $\ell\geq 2$ and $\delta > 0$, then Theorem \ref{main_general_loop} recovers \cite[Theorem 3.13]{EL}. More precisely, for any $r\leq 1$ we compute 
    \[
    \dim HH^{r}(\mathscr{C}_{w})=\ell
    \]
    In fact, for $k\leq 0$ and $\ell\geq 2$ we have $\tilde{\eta}=\ell-2$, and that:
    \begin{itemize}
        \item $|\tilde{\mathcal{Y}}_k|=1$ if $(1-k)\equiv 0 \mod \delta+1$. Otherwise, $|\tilde{\mathcal{Y}}_k|=0$
        \item  $|\tilde{\mathcal{X}}_k|=\begin{cases}
            1& \text{if $(1-k)\equiv 0 \mod (\delta+1)$ or $(1-k)\equiv 1 \mod (\delta+1)$ }\\
            0 & \text{otherwise}
        \end{cases}$
        \item $|\tilde{\mathcal{W}}_k|= \begin{cases}
             1& \text{if $(1-k)\equiv 0 \mod (\delta+1)$ or $(1-k)\equiv 1 \mod (\delta+1)$} \\
             2 & \text{otherwise}
        \end{cases}$
        \subitem here $k<0$
        \item $|\tilde{\mathcal{Z}}_{j,k}|=0$ if $\ell=2$
        \end{itemize}
Moreover, assuming $c\geq 3$ we further have:
        \begin{itemize}
        \item $|\tilde{\mathcal{Z}}_{j,0}|=\begin{cases}
            1& \text{if $j\leq \ell-2$}\\
            0 & \text{if $j>\ell-2$}
            \end{cases}$
            \item $|\tilde{\mathcal{Z}}_{j,k}|=1$ if $(1-k)\equiv q \mod (\delta+1)$ and $(q-1)(\ell-1)<j<q(\ell-1)$, where $0< q< \delta+1$. Otherwise, $|\tilde{\mathcal{Z}}_{j,k}|=0$.
        \end{itemize}

        Thus, $\sum_{j=1}^{d-1}|\tilde{\mathcal{Z}}_{j,k}|=\begin{cases}
        0& \text{if $(1-k)\equiv 0 \mod (\delta+1)$}\\
            \ell-2& \text{otherwise} 
        \end{cases}$.
\end{exe}

\begin{exe}
    If $c=2\ell-1$ and $d=2\ell$ for some integer $\ell\geq 2$, then  writing $(2\ell-2)(1-k) \equiv q \mod (4\ell-3)$ with $0\leq q<4\ell-3$ we have

\begin{itemize}
\item $|\tilde{\mathcal{X}}_k|=\begin{cases}
              1 & \text{if $0\leq q \leq 2(\ell-1)$} \\
              0 & \text{if $q> 2(\ell-1)$}
          \end{cases}$
          \item $|\tilde{\mathcal{W}}_0|=1$
           \item $|\tilde{\mathcal{W}}_k|=\begin{cases}
               2 & \text{if $q=2\ell$}\\
               1 & \text{otherwise}
           \end{cases}$ 
    \item $\sum_{j=1}^{d-1}|\tilde{\mathcal{Z}}_{j,k}|=\begin{cases}
    0 & \text{if $q\in \{0,1\}$}\\
        \min\{2\ell-1,q-1\}-\max\{1,q-2\ell+3\}+1 & \text{if $q>1$}
        \end{cases}$
        \end{itemize}
In particular, when $k=0$ we obtain $\dim HH^{0}(\mathscr{C}_{w})=\dim HH^{1}(\mathscr{C}_{w})=2\ell-1$, whereas when $k=-1$ we get 
    $\dim HH^{-1}(\mathscr{C}_{w})=\dim HH^{-2}(\mathscr{C}_{w})=2$.
\end{exe}

\end{sloppypar}
 
\bibliographystyle{plain}
\bibliography{references}

\date{today}

\end{document}